\documentclass[11pt]{article}
\usepackage{amsmath,amsfonts,amssymb,amsthm}
\usepackage{enumerate}
\usepackage[pagebackref,colorlinks,citecolor=blue,linkcolor=blue]{hyperref}
\usepackage{amsmath,bbm,amssymb,amsxtra}
\usepackage{caption}

\usepackage{geometry}
\numberwithin{equation}{section}
\theoremstyle{plain}
\newtheorem{theorem}[equation]{Theorem}
\newtheorem{corollary}[equation]{Corollary}
\newtheorem{lemma}[equation]{Lemma}
\newtheorem{proposition}[equation]{Proposition}

\theoremstyle{definition}
\newtheorem{definition}[equation]{Definition}

\newtheorem{example}[equation]{Example}
\newtheorem{remark}[equation]{Remark}

\numberwithin{equation}{section}

\newcommand{\bx}{{\partial \Omega}}
\newcommand{\rarrow}{{\rightarrow}}
\newcommand{\Z}{{\mathbb Z}}
\newcommand{\R}{{\mathbb R}}
\newcommand{\N}{{\mathbb N}}

\newcommand{\dyadic}{{\mathcal {Q}}}

\newcommand{\Om}{\Omega}

\providecommand{\vint}[1]{\mathchoice
          {\mathop{\vrule width 5pt height 3 pt depth -2.5pt
                  \kern -9pt \kern 1pt\intop}\nolimits_{\kern -5pt{#1}}}
          {\mathop{\vrule width 5pt height 3 pt depth -2.6pt
                  \kern -6pt \intop}\nolimits_{\kern -3pt{#1}}}
          {\mathop{\vrule width 5pt height 3 pt depth -2.6pt
                  \kern -6pt \intop}\nolimits_{\kern -3pt{#1}}}
          {\mathop{\vrule width 5pt height 3 pt depth -2.6pt
                  \kern -6pt \intop}\nolimits_{\kern -3pt{#1}}}}

\newcommand{\eps}{\varepsilon}
\newcommand{\loc}{\mathrm{loc}}

\newcommand{\BV}{\mathrm{BV}}
\newcommand{\SBV}{\mathrm{SBV}}
\newcommand{\liploc}{\mathrm{Lip}_{\mathrm{loc}}}

\newcommand{\ch}{\text{\raise 1.3pt \hbox{$\chi$}\kern-0.2pt}}

\DeclareMathOperator{\dist}{dist}
\DeclareMathOperator{\diam}{diam}

\DeclareMathOperator{\Lip}{Lip}

\DeclareMathOperator{\supp}{spt}

\begin{document}
\title{Traces of Newton-Sobolev, Haj\l asz-Sobolev,\\
		and BV functions on metric spaces
\footnote{
\hfill \break {\it Keywords}: boundary trace, function of bounded variation, Newton-Sobolev function, Hajlasz-Sobolev function, metric measure space
}
}
\author{Panu Lahti, Xining Li, and Zhuang Wang}
\date{}

\maketitle

\begin{abstract}
We study the boundary traces of Newton-Sobolev, Haj\l asz-Sobolev,
and BV (bounded variation) functions. Assuming less regularity of the domain
than is usually done in the literature,
we show that all of these function classes achieve the same
``boundary values'', which
in particular implies that the trace spaces coincide provided that they exist.
Many of our results seem to be new even in Euclidean spaces but
we work in a more general complete metric space equipped with a doubling measure
and supporting a Poincar\'e inequality.
\end{abstract}

\section{Introduction}

Boundary traces for various function classes, especially functions of
bounded variation (BV functions), have been studied in recent years in the setting
of metric measure spaces $(X,d,\mu)$.
In \cite{LS18}, the authors studied the boundary traces, or traces for short,
of BV functions in suitably
regular domains. Typically, the boundary trace $Tu$ of a function $u$ in a domain
$\Om$ is defined by the condition
\begin{equation}\label{Intro-trace}
\lim_{r\to 0^+}\vint{B(x,r)\cap \Om}|u-Tu(x)|\,d\mu=0
\end{equation}
for a.e. $x\in\partial \Om$ with respect to the
codimension 1 Hausdorff measure $\mathcal{H}$. 
In \cite{MShS} (see also references therein for previous works in Euclidean
spaces) the authors considered the corresponding extension problem,
that is, the problem of finding a function whose trace is a prescribed
{$L^1$-function}
on the boundary.
They showed that in sufficiently regular domains,
the trace operator of BV functions is surjective,
and that in fact the extension can always be taken to be a Newton-Sobolev
function. This implies that
the trace space of both $\BV(\Om)$ and $N^{1,1}(\Om)$ is $L^1(\bx)$.
{This trace and extension problem is motivated by
	Dirichlet problems for functions of
	least gradient, in which one minimizes the total variation among BV functions
	with prescribed boundary data, see \cite{AG78,Giu84,KLLS,M67,SWZ}.}

In the current paper, we consider boundary traces from a different viewpoint.
Unlike in the existing literature, we assume very little regularity of the domain,
meaning that traces need not always exist.
We are nonetheless able to show in various cases that
for a given function, it is possible to
find a more regular function that ``achieves the same boundary values''.
In particular, if the original function has a boundary trace, then the more regular
function has the same trace. This sheds further light on the extension problem.
To prove our results, we apply some existing approximation results for BV and Newton-Sobolev functions, and develop some new ones.


We will always assume that $(X,d,\mu)$ is a complete metric space equipped with
a doubling measure $\mu$ and supporting a $(1,1)$-Poincar\'e inquality.
Let $\Om\subset X$ be a nonempty open set.
For BV functions we prove the following three theorems. The exponent $s$
is sometimes called the homogeneous dimension of the space.
 $N^{1,1}(\Om)$ is a generalization of the Sobolev class $W^{1,1}(\Om)$
to metric spaces;
see Section \ref{sec:preliminaries} for definitions.

\begin{theorem}\label{thm:BV theorem 1}
Let $u\in\BV(\Om)$.
Then there exists $v\in N^{1,1}(\Om)\cap \liploc(\Om)$ such that
\[
\vint{B(x,r)\cap \Om}|v-u|^{s/(s-1)}\,d\mu\to 0\quad \textrm{as }r\to 0^+
\]
uniformly for all $x\in\partial\Om$.
\end{theorem}

In particular, \emph{whenever} there exists a BV extension of a given
function defined on the boundary, it is possible
to also find a Newton-Sobolev extension.
If we give up the requirement that $v$ is locally Lipschitz, we can
replace $s/(s-1)$ by an arbitrarily large exponent.

\begin{theorem}\label{thm:BV theorem 2}
Let $u\in\BV(\Om)$ and let $1\le q<\infty$.
Then there exists $v\in N^{1,1}(\Om)$ such that
\[
\vint{B(x,r)\cap \Om}|v-u|^{q}\,d\mu\to 0\quad \textrm{as }r\to 0^+
\]
uniformly for all $x\in\partial\Om$.
\end{theorem}

If we also allow $v$ to have a small (approximate) \emph{jump set} $S_v$, then
we can include the case $q=\infty$. The class of
\emph{special functions of bounded variation},
denoted by $\SBV(\Om)$, is defined as those $\BV$ functions whose variation measure
only has an absolutely continuous part (like Sobolev functions) and a jump part.
The class was introduced by De Giorgi and Ambrosio
\cite{ADG} as a natural class in which to solve
various variational problems, e.g. the Mumford--Shah functional.

\begin{theorem}\label{thm:BV theorem 3}
Let $u\in\BV(\Om)$ and let $\eps>0$. Denote $\Om(r):=\{x\in \Om:\,\dist(x,X\setminus \Om)>r\}$ for $r>0$. Then there exists $v\in \SBV(\Om)$
such that $\mathcal H(S_v)<\eps$ and
\[
\Vert v-u\Vert_{L^{\infty}(\Om\setminus \Om(r))}\to 0\quad \textrm{as }r\to 0^+.
\]
\end{theorem}

Note that $v\in \SBV(\Om)$ belongs to $N^{1,1}(\Om)$ if and only if
$\mathcal H(S_v)=0$
{(see \cite[Theorem 4.1]{KKST}, \eqref{eq:variation measure decomposition}, and \cite[Theorem 4.6]{HKLL}).}
Thus we could equivalently require
\begin{itemize}
\item $v\in \SBV(\Om)\cap \liploc(\Om)$ (in particular, $S_v=\emptyset)$ in Theorem \ref{thm:BV theorem 1},
\item $v\in \SBV(\Om)$ with $\mathcal H(S_v)=0$ in Theorem \ref{thm:BV theorem 2}, and
\item $v\in \SBV(\Om)$ with $\mathcal H(S_v)<\eps$ in Theorem \ref{thm:BV theorem 3},
\end{itemize}
illustrating how we get better boundary approximation
by relaxing the regularity requirements on $v$.

From Theorem \ref{thm:BV theorem 1} (or Theorem \ref{thm:BV theorem 2}), we obtain the following corollary.

\begin{corollary}\label{BV-Sobolev}
The trace spaces of $\BV(\Om)$ and $N^{1,1}(\Om)$ are the same.
\end{corollary}

The definitions of  trace and trace space are given in Definition \ref{trace}
and Definition \ref{trace-space}. Here and throughout this paper, for two
Banach function spaces
{$\mathbb X(\Om)$ and $\mathbb Y(\Om)$,}
that the trace
spaces of $\mathbb X(\Omega)$ and $\mathbb Y(\Omega)$ are the same means that
if the Banach function space $\mathbb Z(\bx)$ is the trace space of
$\mathbb X(\Omega)$, then it is also the trace space of
$\mathbb Y(\Omega)$, and vice versa.

Corollary \ref{BV-Sobolev} is stronger than we
{expected; it says that we can}
obtain the existence of the trace and the trace space of $\BV(\Omega)$
by only knowing the existence of the trace and the trace space
of $N^{1,1}(\Omega)$, which is nontrivial, since $N^{1,1}(\Omega)$
is a strict subset of $\BV(\Omega)$.

The so-called Haj\l asz-Sobolev space $M^{1,p}(\Om)$, $p\geq 1$,
introduced in \cite{H96}, is a subspace of $N^{1,p}(\Om)$.
For $p>1$ and $\Om$ supporting a $(1,p)$-Poincar\'e inequality and
a doubling measure,
we have $N^{1,p}(\Omega)=M^{1, p}(\Omega)$ with equivalent norms, see
\cite{H03}, and hence
the traces of $M^{1,p}(\Omega)$ and $N^{1,p}(\Omega)$ will be the same.
But for $p=1$, even under these strong assumptions, $M^{1,1}(\Om)$
is only a strict subspace of $N^{1,1}(\Om)$ and it seems that
trace results for $M^{1,1}$ are lacking in the literature. One can also define a
local version $M^{1,1}_{c_H}(\Om)$, see Section \ref{sec:preliminaries} and
Remark \ref{M_1} for more information.
For these classes, we prove the following results.

\begin{theorem}\label{Sobolev-H}
Suppose $\Omega$ satisfies the measure density condition \eqref{measure-density}. Then there exists $0<c_H<1$ such that for any $u\in N^{1,1}(\Omega)$, there is $v\in M^{1,1}_{c_H}(\Omega)\cap \Lip_{\rm loc}(\Omega)$ satisfying $\|v\|_{M^{1,1}_{c_H}} (\Omega) \lesssim \|u\|_{N^{1,1}(\Omega)}$ and 
$$\lim_{r\rarrow 0^+} \vint{B(x, r)\cap \Omega} |v-u|\, d\mu=0$$
for $\mathcal{H}$-a.e. $x\in \partial \Omega$, where $\mathcal H$ is the codimension $1$ Hausdorff measure.

If additionally $\Omega$ is a uniform domain, then $v$ can be chosen in $M^{1,1}(\Omega)\cap \Lip_{\rm loc}(\Omega)$.
\end{theorem}

With the exception of this theorem, our results are not written
in terms of the codimension 1 Hausdorff measure $\mathcal H$ (defined in \eqref{def-codimension} and \eqref{def-codimension-1}) which is
used in most existing literature.
In Theorems \ref{thm:BV theorem 1}--\ref{thm:BV theorem 3}, 
the results hold for every point on the boundary.
On the other hand, the space or domain may be endowed with a measure $\mu$
for which the codimension $1$ Hausdorff measure is not $\sigma$-finite
on the boundary of the domain (see Example \ref{ex:weight}). More precisely,
in Example \ref{ex:weight} we define a weighted measure on the Euclidean
half-space $\R^2_{+}$ whose codimension $1$ Hausdorff measure is infinity
for any open interval of $\partial \R^{2}_+=\R$. But on $\R^2_{+}$, it
is natural to study instead the trace with respect to the $1$-dimensional
Lebesgue measure on $\R$,
which we do in Example \ref{dyadic-trace}.
{Another motivation for us is that in certain Dirichlet problems
	one needs to consider the trace with respect to a measure different from
	$\mathcal H$, see \cite[Definition 4.1]{KLLS}.}

{More generally,}
instead of only studying the codimension $1$ Hausdorff measure, we may study any arbitrary boundary measure $\widetilde{\mathcal H}$ on $\bx$. 
In order to study such problems,
we first replace
the codimension $1$ Hausdorff measure $\mathcal H$ with
$\widetilde{\mathcal H}$ in the previous definition of trace to give
the definition of trace with respect to $\widetilde{\mathcal H}$, see
Definition \ref{trace-space-measure}.
Then we prove the following result.

\begin{theorem}\label{Sobolev-H-measure}
Suppose $\Omega$ satisfies the measure doubling condition  \eqref{measure-doubling}.
Let $\widetilde {\mathcal H}$ be any Radon measure on $\partial \Omega$.
Suppose that for a given $u\in N^{1,1}(\Omega)$, there exists a function $Tu$ such that
$$\lim_{r\rarrow 0^+} \vint{B(x, r)\cap \Omega} |u-Tu(x)|\, d\mu=0$$
for $\widetilde {\mathcal H}$-a.e. $x\in \partial\Omega$. Then
there exist  $0<c_H<1$ and
$v\in M^{1,1}_{c_H}(\Omega)\cap \Lip_{\rm loc}(\Omega)$ such that
$\|v\|_{M^{1,1}_{c_H}} (\Omega) \lesssim \|u\|_{N^{1,1}(\Omega)}$ and 
$$\lim_{r\rarrow 0^+} \vint{B(x, r)\cap \Omega} |v-Tu(x)|\, d\mu=0$$
for $\widetilde{\mathcal{H}}$-a.e. $x\in \partial \Omega$.

If additionally $\Omega$ is a uniform domain, then $v$ can be chosen in $M^{1,1}(\Omega)\cap \Lip_{\rm loc}(\Omega)$.
\end{theorem}

Similarly to Corollary \ref{BV-Sobolev},
from Theorem \ref{Sobolev-H} and Theorem \ref{Sobolev-H-measure} we obtain
the following corollary.

\begin{corollary}\label{SObolev-M-2}
Let $\Om\subset X$ be a uniform domain
and suppose that $\Omega$ satisfies the measure doubling condition
\eqref{measure-doubling}. Then for any given boundary measure
$\widetilde {\mathcal H}$, the trace spaces of $N^{1,1}(\Om)$ and $M^{1,1}(\Om)$
with respect to any boundary measure $\widetilde {\mathcal H}$ on $\bx$ are the same.
\end{corollary}

The paper is organized as follows. In Section \ref{sec:preliminaries}, we give the necessary preliminaries. In Section 3, we study the traces of $N^{1,1}$ and
$\BV$ and give the proofs of Theorems
\ref{thm:BV theorem 1}--\ref{thm:BV theorem 3} and Corollary \ref{BV-Sobolev}.
In Section 4, we study the traces of $N^{1,1}$ and $M^{1,1}$ and
give the proofs of Theorem \ref{Sobolev-H}, Theorem \ref{Sobolev-H-measure},
and Corollary \ref{SObolev-M-2}. Finally, in Section 5, apart from giving
several examples that we refer to in Section 3 and Section 4, we also discuss
some trace results and examples obtained as applications of
Corollary \ref{BV-Sobolev}
and Corollary \ref{SObolev-M-2}.

\bigskip

{\noindent\bf Acknowledgments.}
The authors would like to thank Pekka Koskela
and Nageswari Shanmugalingam for reading the manuscript and giving comments that helped improve the paper.  

X. L. is supported by NNSF of China (No. 11701582). Z. W. is supported by the Academy of Finland via  Centre of Excellence in Analysis and Dynamics Research (No. 307333) and partially supported by the grant 346300 for IMPAN from the Simons Foundation and the matching 2015-2019 Polish MNiSW fund.

\section{Preliminaries}\label{sec:preliminaries}

In this section we introduce the notation, definitions,
and assumptions used in the paper.

Throughout this paper, $(X,d,\mu)$ is a complete {}metric space that is equip\-ped
with a metric $d$ and a Borel regular outer measure $\mu$ satisfying
a doubling property, meaning that
there exists a constant $C_d\ge 1$ such that
\[
0<\mu(B(x,2r))\le C_d\mu(B(x,r))<\infty
\]
for every ball $B(x,r):=\{y\in X:\,d(y,x)<r\}$.
By iterating the doubling condition, for every $0<r\leq R$ and $y\in B(x, R)$, we have 
\begin{equation}\label{eq:homogeneous dimension}
\frac{\mu(B(y, r))}{\mu(B(x, R))}\geq 4^{-s} \left(\frac{r}{R}\right)^s,
\end{equation}
{for any $s\ge\log_2 C_d$. See \cite[Lemma 4.7]{H03} or \cite{BB}
for a proof of this.
We fix such an $s>1$ and call it}
the {\it  homogeneous dimension}. 

 The letters $c,C$ (sometimes with a subscript) will denote positive constants
that usually depend only on the space and may change at
different occurrences; if $C$ depends on $a,b,\ldots$, we write
$C=C(a,b,\ldots)$.
The notation $A\approx B$ means that there is a constant
$C$ such that $1/C\cdot A\leq B\leq C\cdot A$. The notation $A\lesssim B$
($A\gtrsim B$) means that there is a constant $C$ such that
$A\leq  C\cdot B$ ($A\geq C\cdot B$).

All functions defined on $X$ or its subsets will take values in $[-\infty,\infty]$.
A complete metric space equipped with a doubling measure is proper,
that is, closed and bounded sets are compact.
For an open set $\Omega\subset X$,
a function is in the class $L^1_{\loc}(\Omega)$ if and only if it is in $L^1(\Om')$ for
every open $\Omega'\Subset\Omega$.
Here $\Omega'\Subset\Omega$ means that $\overline{\Omega'}$ is a
compact subset of $\Omega$.
Other local spaces of functions are defined similarly.

For any set $A\subset X$ and $0<R<\infty$, the restricted spherical Hausdorff content
of codimension $1$ is defined as
\begin{equation}\label{def-codimension}
\mathcal{H}_{R}(A):=\inf\left\{\sum_{j\in I}
\frac{\mu(B(x_{j},r_{j}))}{r_{j}}:\,A\subset\bigcup_{j\in I}B(x_{j},r_{j}),
\,r_{j}\le R,\,I\subset\N\right\}.
\end{equation}
The codimension $1$ Hausdorff measure of $A\subset X$ is then defined as
\begin{equation}\label{def-codimension-1}
\mathcal{H}(A):=\lim_{R\rightarrow 0^+}\mathcal{H}_{R}(A).
\end{equation}

Given an open set $\Om\subset X$, we can regard it as a metric space in its own right, equipped with the metric induced by $X$ and the measure $\mu|_{\Omega}$ which is the restriction of $\mu$ to subsets of $\Om$. This restricted measure $\mu|_{\Omega}$ is a Radon measure, see \cite[Lemma 3.3.11]{HKST}.

We say that an open set $\Omega$ satisfies a {\it measure density condition} if there is a constant $c_m>0$ such that 
\begin{equation}\label{measure-density}
\mu(B(x, r)\cap \Omega)\geq c_m \mu(B(x, r))
\end{equation}
{for every $x\in \overline{\Om}$ and every $r\in (0, \diam (\Omega))$.}
We say that
$\Omega$ satisfies a {\it measure doubling condition} if the measure $\mu|_{\Omega}$ is a doubling measure, i.e., there is a constant $c_d>0$ such that 
 \begin{equation}\label{measure-doubling}
 0<\mu(B(x, 2r)\cap \Om)\leq c_d \mu(B(x, r)\cap\Om)<\infty
 \end{equation}
 for every $x\in \overline{\Om}$
and every $r>0$. Notice that if $\Omega$ satisfies the measure density condition,
then it satisfies the measure doubling condition.

By a curve we mean a rectifiable continuous mapping from a compact interval
of the real line into $X$.
A nonnegative Borel function $g$ on $X$ is an upper gradient 
of a function $u$
on $X$ if for all nonconstant curves $\gamma$, we have
\begin{equation}\label{eq:definition of upper gradient}
|u(x)-u(y)|\le \int_\gamma g\,ds,
\end{equation}
where $x$ and $y$ are the end points of $\gamma$
and the curve integral is defined by using an arc-length parametrization,
see \cite[Section 2]{HK} where upper gradients were originally introduced.
We interpret $|u(x)-u(y)|=\infty$ whenever  
at least one of $|u(x)|$, $|u(y)|$ is infinite.

We say that a family of curves $\Gamma$ is of zero $1$-modulus if there is a 
nonnegative Borel function $\rho\in L^1(X)$ such that 
for all curves $\gamma\in\Gamma$, the curve integral $\int_\gamma \rho\,ds$ is infinite.
A property is said to hold for $1$-almost every curve
if it fails only for a curve family with zero $1$-modulus. 
If $g$ is a nonnegative $\mu$-measurable function on $X$
and (\ref{eq:definition of upper gradient}) holds for $1$-almost every curve,
we say that $g$ is a $1$-weak upper gradient of $u$. 
By only considering curves $\gamma$ in $A\subset X$,
we can talk about a function $g$ being a ($1$-weak) upper gradient of $u$ in $A$.

Given a $\mu$-measurable set $H\subset X$, we let
\[
\Vert u\Vert_{N^{1,1}(H)}:=\Vert u\Vert_{L^1(H)}+\inf \Vert g\Vert_{L^1(H)},
\]
where the infimum is taken over all $1$-weak upper gradients $g$ of $u$ in $H$.
The substitute for the Sobolev space $W^{1,1}$ in the metric setting is the Newton-Sobolev space
\[
N^{1,1}(H):=\{u:\|u\|_{N^{1,1}(H)}<\infty\},
\]
which was first introduced in \cite{S}.
It is known that for any $u\in N_{\loc}^{1,1}(H)$ there exists a minimal $1$-weak
upper gradient of $u$ in $H$, always denoted by $g_{u}$, satisfying $g_{u}\le g$ 
$\mu$-a.e. in $H$, for any $1$-weak upper gradient $g\in L_{\loc}^{1}(H)$
of $u$ in $H$, see \cite[Theorem 2.25]{BB}.

Next we present the basic theory of functions
of bounded variation on metric spaces. This was first developed in
\cite{A1, M}; see also the monographs \cite{AFP, EvaG92, Fed, Giu84, Zie89} for the classical 
theory in Euclidean spaces.
We will always denote by $\Om$ an open subset of $X$.
Given a function $u\in L^1_{\loc}(\Om)$,
we define the total variation of $u$ in $\Om$ by
\begin{equation}\label{eq:total variation}
\|Du\|(\Om):=\inf\left\{\liminf_{i\to\infty}\int_\Om g_{u_i}\,d\mu:\,
{u_i\in N^{1,1}_{\loc}(\Om),}\,
u_i\to u\textrm{in } L^1_{\loc}(\Om)\right\},
\end{equation}
where each $g_{u_i}$ is the minimal $1$-weak upper gradient of $u_i$
in $\Om$.
(In \cite{M}, local Lipschitz constants were used in place of upper gradients, but the theory
can be developed similarly with either definition.)
We say that a function $u\in L^1(\Om)$ is of bounded variation, 
and denote $u\in\BV(\Om)$, if $\|Du\|(\Om)<\infty$.
For an arbitrary set $A\subset X$, we define
\[
\|Du\|(A):=\inf\{\|Du\|(W):\, A\subset W,\,W\subset X
\text{is open}\}.
\]

\begin{proposition}[{\cite[Theorem 3.4]{M}}]\label{thm:variation measure property}
If $u\in L^1_{\loc}(\Om)$, then $\|Du\|(\cdot)$ is
a Borel measure on $\Omega$.
\end{proposition}

For any $u,v\in L^1_{\loc}(\Om)$,
it is straightforward to show that
\begin{equation}\label{eq:BV functions form vector space}
\Vert D(u+v)\Vert(\Om)\le \Vert Du\Vert(\Om)+\Vert Dv\Vert(\Om).
\end{equation}

The BV norm is defined by
\[
\Vert u\Vert_{\BV(\Om)}:=\Vert u\Vert_{L^1(\Om)}+\Vert Du\Vert(\Om).
\]

We will assume throughout the paper that $X$ supports a $(1,1)$-Poincar\'e inequality,
meaning that there exist constants $C_P>0$ and $\lambda \ge 1$ such that for every
ball $B(x,r)$, every $u\in L^1_{\loc}(X)$,
and every upper gradient $g$ of $u$,
we have
\[
\vint{B(x,r)}|u-u_{B(x,r)}|\, d\mu 
\le C_P r\vint{B(x,\lambda r)}g\,d\mu,
\]
where 
\[
u_{B(x,r)}:=\vint{B(x,r)}u\,d\mu :=\frac 1{\mu(B(x,r))}\int_{B(x,r)}u\,d\mu.
\]
Recall the exponent $s>1$ from \eqref{eq:homogeneous dimension}.
The $(1,1)$-Poincar\'e inequality implies the so-called Sobolev-Poincar\'e inequality, see e.g. \cite[Theorem 4.21]{BB}, and by applying the latter to approximating locally Lipschitz functions in the definition of the total variation, we get the following Sobolev-Poincar\'e inequality for $\BV$ functions. For every ball $B(x,r)$ and every $u\in L^1_{\loc}(X)$, we have
\begin{equation}\label{eq:sobolev poincare inequality}
\left(\,\vint{B(x,r)}|u-u_{B(x,r)}|^{s/(s-1)}\,d\mu\right)^{(s-1)/s}
\le C_{SP}r\frac{\Vert Du\Vert (B(x,2\lambda r))}{\mu(B(x,2\lambda r))},
\end{equation}
where $C_{SP}=C_{SP}(C_d,C_P,\lambda)\ge 1$ is a constant.

For an open set $\Omega\subset X$ and a $\mu$-measurable set $E\subset X$ with
$\Vert D\ch_E\Vert(\Om)<\infty$, we know that for any Borel set $A\subset\Omega$,
\begin{equation}\label{eq:def of theta}
\Vert D\ch_E\Vert(A)=\int_{\partial^{*}E\cap A}\theta_E\,d\mathcal H,
\end{equation}
where
$\theta_E\colon X\to [\alpha,C_d]$ with $\alpha=\alpha(C_d,C_P,\lambda)>0$, see \cite[Theorem 5.3]{A1} 
and \cite[Theorem 4.6]{AMP}.
The following coarea formula is given in \cite[Proposition 4.2]{M}:
if $\Omega\subset X$ is an open set and $u\in L^1_{\loc}(\Omega)$, then
\begin{equation}\label{eq:coarea}
\|Du\|(\Omega)=\int_{-\infty}^{\infty}P(\{u>t\},\Omega)\,dt.
\end{equation}

The lower and upper approximate limits of a function $u$ on $\Om$ are defined respectively by
\[
u^{\wedge}(x):
=\sup\left\{t\in\R:\,\lim_{r\to 0}\frac{\mu(\{u<t\}\cap B(x,r))}{\mu(B(x,r))}=0\right\}
\]
and
\[
u^{\vee}(x):
=\inf\left\{t\in\R:\,\lim_{r\to 0}\frac{\mu(\{u>t\}\cap B(x,r))}{\mu(B(x,r))}=0\right\}.
\]
Then the jump set $S_u$ is defined as the set of points $x\in\Om$
for which $u^{\wedge}(x)<u^{\vee}(x)$.
It is straightforward to check that $u^{\wedge}$ and $u^{\vee}$ are Borel functions.

By \cite[Theorem 5.3]{AMP}, the variation measure of a $\BV$ function
can be decomposed into the absolutely continuous and singular part, and the latter
into the Cantor and jump part, as follows. Given an open set 
$\Omega\subset X$ and $u\in\BV(\Omega)$, we have for any Borel set $A\subset \Om$
\begin{equation}\label{eq:variation measure decomposition}
\begin{split}
\Vert Du\Vert(A) &=\Vert Du\Vert^a(A)+\Vert Du\Vert^s(A)\\
&=\Vert Du\Vert^a(A)+\Vert Du\Vert^c(A)+\Vert Du\Vert^j(A)\\
&=\int_{A}a\,d\mu+\Vert Du\Vert^c(A)+\int_{A\cap S_u}\int_{u^{\wedge}(x)}^{u^{\vee}(x)}\theta_{\{u>t\}}(x)\,dt\,d\mathcal H(x),
\end{split}
\end{equation}
where $a\in L^1(\Omega)$ is the density of the absolutely continuous part
$\Vert Du\Vert^a(A)$ of $\Vert Du\Vert(A)$
and the functions $\theta_{\{u>t\}}\in [\alpha,C_d]$ 
are as in~\eqref{eq:def of theta}.

Next, we introduce the Haj\l asz-Sobolev space.
{Let $0< p<\infty$.}
Given a $\mu$-measurable set
$K\subset X$, we define $M^{1,p}(K)$ to be the set of all functions $u\in L^p(K)$
for which there exists $0\leq g\in L^p(K)$ and a set $A\subset K$ of measure
zero such that for all $x, y\in K\setminus A$ we have the estimate
\begin{equation}\label{H-gradient}
|u(x)-u(y)|\leq d(x, y)(g(x)+g(y)).
\end{equation}
The corresponding norm
(when $p\ge 1$)
is obtained by setting
$$\|u\|_{M^{1,p}(K)}=\|u\|_{L^p(K)}+\inf \|g\|_{L^p(K)},$$
where the infimum is taken over all admissible functions $g$ in \eqref{H-gradient}.
We refer to \cite{H96,H03} for more properties of the Haj\l asz-Sobolev
space $M^{1,p}$.
The space $M^{1,p}_{c_H}(K)$  is defined exactly in the same manner as the
space $M^{1,p}(K)$ except for one difference: in the definition of
$M^{1,p}_{c_H}(K)$, the condition \eqref{H-gradient} is assumed to
hold only for points $x, y\in K\setminus A$ that satisfy the condition
\begin{equation}\label{H-gradient-ball}
d(x, y)\leq  c_H\cdot \min\{d(x, X\setminus K), d(y, X\setminus K)\},
\end{equation}
where $0<c_H<1$ is a constant.

We give the following definitions for the boundary trace, or trace for short, of a function defined on an open set $\Om$.

\begin{definition}\label{trace}
Let $\Om\subset X$ be an open set and let $u$ be a $\mu$-measurable function on $\Om$. A number $Tu(x)$ is the trace of $u$ at $x\in \bx$ if we have
\begin{equation}\label{def-trace}
\lim_{r\rarrow 0^+} \vint{B(x, r)\cap \Om} |u- Tu(x)|\, d\mu=0.
\end{equation}
We say that $u$ has a trace $Tu$ in $\partial \Om$
if $Tu(x)$ exists for $\mathcal H$-almost every $x\in \bx$.
\end{definition}

Moreover, we give the following definitions for the trace space of a Banach space defined on an open set $\Om$.

\begin{definition}\label{trace-space}
Let $\Om$ be an open set and let $\mathbb X(\Om)$ be a Banach function space on $\Om$. A Banach space $\mathbb Y(\bx, \mathcal H)$ on $\bx$ is the trace space of $\mathbb X(\Om)$ if the trace operator $u\mapsto Tu$ defined  in Definition \ref{trace} is a bounded linear surjective operator from $\mathbb  X(\Om)$ to $\mathbb Y(\bx, \mathcal H)$.
\end{definition}

\begin{definition}\label{trace-space-measure}
Let $\Omega$ be an open set and $\widetilde {\mathcal H}$ be a measure on $\bx$.
Let $\mathbb X(\Omega)$ be a Banach function space on $\Omega$. A Banach space
$\mathbb Y(\bx, \widetilde{\mathcal H})$ on $\bx$ is the trace space of
$\mathbb X(\Omega)$ with respect to $\widetilde{\mathcal H}$, if the trace operator
$u\mapsto Tu$ defined in Definition \ref{trace} by replacing $\mathcal H$
by $\widetilde{\mathcal H}$ is a bounded linear surjective operator from
$\mathbb X(\Om)$ to $\mathbb Y(\bx, \widetilde{\mathcal H})$.
\end{definition}

\section{Traces of $N^{1,1}(\Om)$ and $\BV(\Om)$}

In this section, let $\Om\subset X$ be an arbitrary nonempty
open set. Recall the definition
of the number $s>1$ from \eqref{eq:homogeneous dimension}.

\begin{lemma}\label{lem:approximating liploc functions in L with Sobolev exponent}
Let $u\in L^1_{\loc}(\Om)$ with $\Vert Du\Vert(\Om)<\infty$. Then there exists
a sequence $(u_i)\subset \liploc(\Om)$ such that $u_i\to u$ in $L_{\loc}^{s/(s-1)}(\Om)$
and
\[
\Vert Du\Vert(\Om)=\lim_{i\to\infty}\int_{\Om}g_{u_i}\,d\mu.
\]
\end{lemma}
\begin{proof}
By the Sobolev-Poincar\'e inequality \eqref{eq:sobolev poincare inequality},
we have $u\in L^{s/(s-1)}_{\loc}(\Om)$.
Take open sets
$\Om_1\Subset \Om_2\Subset \ldots \Subset \Om=\bigcup_{j=1}^{\infty}\Om_j$.
Now $u\in L^{s/(s-1)}(\Om_j)$ for each $j\in\N$.
Define the truncations
\[
u_M:=\min\{M,\max\{-M,u\}\},\quad M>0.
\]
For each $j\in\N$ we find a number $M_j>0$ such that
$\Vert u_{M_j}-u\Vert_{L^{s/(s-1)}(\Om_j)}<1/j$.
From the definition of the total variation,
take a sequence $(v_i)\subset \liploc(\Om)$ such that $v_i\to u$ in $L_{\loc}^1(\Om)$
and
\[
\Vert Du\Vert(\Om)=\lim_{i\to\infty}\int_{\Om}g_{v_i}\,d\mu.
\]
Then also $(v_{i})_{M_j}\to u_{M_j}$ in $L^{s/(s-1)}(\Om_j)$ for all $j\in\N$.
Thus we can pick indices $i(j)\ge j$ such that
$\Vert (v_{i(j)})_{M_j}-u_{M_j}\Vert_{L^{s/(s-1)}(\Om_j)}<1/j$ for each $j\in\N$.
Defining $u_j:=(v_{i(j)})_{M_j}$, we now have
\[
\Vert u_j-u\Vert_{L^{s/(s-1)}(\Om_j)}<2/j\quad\textrm{for all }j\in\N
\]
and so $u_j\to u$ in $L^{s/(s-1)}_{\loc}(\Om)$.
Moreover, since truncation does not increase energy,
\[
\limsup_{j\to\infty}\int_{\Om}g_{u_j}\,d\mu\le \Vert Du\Vert(\Om).
\]
But by lower semicontinuity, also
$\Vert Du\Vert(\Om)\le \liminf_{j\to\infty}\int_{\Om}g_{u_j}\,d\mu$.
\end{proof}

We have the following standard fact; for a proof see e.g.
\cite[Proposition 3.8]{HKLL}.

\begin{lemma}\label{lem:weak* convergence}
Let $u\in L^1_{\loc}(\Om)$ with $\Vert Du\Vert(\Om)<\infty$ and let
$(u_i)\subset N^{1,1}_{\loc}(\Om)$ with $u_i\to u$ in $L^1_{\loc}(\Om)$ and
\[
\Vert Du\Vert(\Om)=\lim_{i\to\infty}\int_{\Om}g_{u_i}\,d\mu.
\]
Then we also have the weak* convergence
$g_{u_i}\,d\mu\overset{*}{\rightharpoonup}d\Vert Du\Vert$.
\end{lemma}

\begin{lemma}\label{lem:choosing approximating function}
	Let $\Om_1\Subset \Om_2\Subset \ldots \Subset\bigcup_{j=1}^{\infty}\Om_j=\Om$ be open sets, let $\Om_0:=\emptyset$, and
	let $\eta_j\in \Lip_c(\Om_{j})$ such that $0\le \eta_j\le 1$ on $X$ and
	$\eta_j=1$ in $\Om_{j-1}$ for each $j\in\N$, with $\eta_1\equiv 0$.
	Let $1\le q<\infty$.
	Moreover, let $u\in L^1_{\loc}(\Om)$ with
	$\Vert Du\Vert(\Om)<\infty$, and for each $j\in\N$
	let $(u_{j,i})\subset N^{1,1}(\Om_j)$ such that $u_{j,i}- u\to 0$
	in $L^{q}(\Om_j)$ and
	\[
	\lim_{i\to\infty}\int_{\Om_j} g_{u_{j,i}}\,d\mu=\Vert Du\Vert(\Om_j),
	\]
	where each $g_{u_{j,i}}$ is the minimal $1$-weak upper gradient of
	$u_{j,i}$ in $\Om_j$.
	Finally, let $\delta_j>0$ for each $j\in\N$, and let $\eps>0$. Then
	for each $j\in\N$ we find an index $i(j)$ such that
	letting $u_{j}:=u_{j,i(j)}$ and
	\[
	v:=\sum_{j=2}^{\infty}(\eta_j-\eta_{j-1})u_{j},
	\]
	we have
	\[
	\max\{\Vert v-u\Vert_{L^{1}(\Om_{j}\setminus \Om_{j-1})},
	\Vert v-u\Vert_{L^{q}(\Om_{j}\setminus \Om_{j-1})}\}<\delta_j
	\quad\textrm{for all }j\in\N,
	\]
	and $\int_\Om g_v\,d\mu<\Vert Du\Vert(\Om)+\eps$.
\end{lemma}

Note that neither $u$ nor the functions $u_{j,i}$ need to be in $L^q(\Om_j)$,
only in $L^1(\Om_j)$, but still we can have
$u_{j,i}- u\to 0$ in $L^{q}(\Om_j)$ for each $j\in\N$.
We can also see that in $\Omega_j\setminus \Omega_{j-1}$, the function
$v$ can be written as the finite sum
(let $\eta_0\equiv 0$)
	\begin{equation}\label{finite sum}
	\sum_{i=2}^{\infty}(\eta_i-\eta_{i-1})u_{i}
	=(\eta_j-\eta_{j-1})u_j+(\eta_{j+1}-\eta_j)u_{j+1}
	=\eta_j u_j+(1-\eta_{j})u_{j+1}.
	\end{equation}

\begin{proof}
	By Lemma \ref{lem:weak* convergence}, for each $j\in\N$ we have
	$g_{u_{j,i}}\,d\mu\overset{*}{\rightharpoonup} d\Vert Du\Vert$
	as $i\to\infty$ in $\Om_j$.
	For each $j\in\N$, let $L_j>0$ denote a Lipschitz constant of $\eta_j$;
	we can take this to be an increasing sequence.
	Set $\delta_0:=1$, $L_0:=1$.
	Letting $u_j:=u_{j,i(j)}$ for suitable indices $i(j)\in\N$, we get
	\begin{equation}\label{eq:choosing the L1 closeness of uis}
	\max\{\Vert u_{j}-u\Vert_{L^{1}(\Om_{j})},\Vert u_{j}-u\Vert_{L^{q}(\Om_{j})}\}
	<\min\{\delta_{j-1},\delta_{j},2^{-j-1}\eps/L_{j}\}/2
	\end{equation}
	for all $j\in\N$, and
	\begin{equation}\label{eq:choosing the energy of uis}
	\int_{\Om_j}(\eta_{j}-\eta_{j-1}) g_{u_j}\,d\mu
	<\int_{\Om_j}(\eta_{j}-\eta_{j-1})\,d\Vert Du\Vert+2^{-j}\eps
	\end{equation}
	for all $j=2,3,\ldots$.
	We get for all $j\in\N$
	\begin{align*}
	\Vert v-u\Vert_{L^{q}(\Om_{j}\setminus \Om_{j-1})}
	&=\Vert \sum_{i=2}^{\infty}(\eta_i-\eta_{i-1})u_{i}-u\Vert_{L^{q}(\Om_{j}\setminus \Om_{j-1})}\\
	&{\stackrel{\eqref{finite sum}}{=}}\Vert \eta_{j}u_{j}+(1-\eta_{j})u_{j+1}-u\Vert_{L^{q}(\Om_{j}\setminus \Om_{j-1})}\\
	&=\Vert \eta_{j}u_{j}+(1-\eta_{j})u_{j+1}
		-\eta_{j}u-(1-\eta_{j})u\Vert_{L^{q}(\Om_{j}\setminus \Om_{j-1})}\\
	&\le \Vert u_{j}-u\Vert_{L^{q}(\Om_{j}\setminus \Om_{j-1})}
	+\Vert u_{j+1}-u\Vert_{L^{q}(\Om_{j}\setminus \Om_{j-1})}\\
	&<\delta_j
	\end{align*}
	by \eqref{eq:choosing the L1 closeness of uis} as desired,
	and similarly for the $L^1$-norm.
	Let $v_2:=u_2$ in $\Om_2$, and recursively
	$v_{i+1}:=\eta_{i} v_i+(1-\eta_{i})u_{i+1}$ in $\Om_{i+1}$.
	We see that $v=\lim_{i\to\infty}v_i$ (at every point in $\Om$).
	By the proof of the Leibniz rule in \cite[Lemma 2.18]{BB},
	the minimal $1$-weak upper gradient
	of $v_3$ in $\Om_3$ satisfies
	\[
	g_{v_3}\le g_{\eta_2}|u_2-u_3|+\eta_2 g_{u_2}+(1-\eta_2)g_{u_3}.
	\]
	Inductively, we get for $i=3,4,\ldots$
	\[
	g_{v_i}\le \sum_{j=2}^{i-1} g_{\eta_{j}}|u_j-u_{j+1}|+
	\sum_{j=2}^{i-1}(\eta_{j}-\eta_{j-1}) g_{u_j}+(1-\eta_{i-1})g_{u_{i}}\quad \textrm{in }\Om_i;
	\]
	to prove this, assume that it holds for the index $i$. Then we have by applying a Leibniz rule as above, and noting that $g_{\eta_i}$ can be nonzero only in
	$\Om_{i}\setminus \Om_{i-1}$ (see \cite[Corollary 2.21]{BB}),
	where $v_i=u_i$,
	\begin{align*}
g_{v_{i+1}}
&\le g_{\eta_{i}}|v_i-u_{i+1}|+\eta_{i}
g_{v_i}+(1-\eta_{i})g_{u_{i+1}}\\
&= g_{\eta_{i}}|u_i-u_{i+1}|+\eta_{i}
	g_{v_i}+(1-\eta_{i})g_{u_{i+1}}\\
&\stackrel{\text
		{\tiny{
		Induction}}}{\le} g_{\eta_{i}}|u_{i}-u_{i+1}|+\sum_{j=2}^{i-1} g_{\eta_{j}}|u_j-u_{j+1}|\\
&\qquad
+\sum_{j=2}^{i-1}(\eta_j-\eta_{j-1})g_{u_{j}}+(\eta_{i}-\eta_{i-1})g_{u_{i}}
+(1-\eta_{i})g_{u_{i+1}}\\
&=\sum_{j=2}^{i} g_{\eta_{j}}|u_j-u_{j+1}|+
\sum_{j=2}^{i}(\eta_{j}-\eta_{j-1}) g_{u_j}+(1-\eta_{i})g_{u_{i+1}}
\quad \textrm{in }\Om_{i+1}.
\end{align*}
This completes the induction.
In each $\Omega_{j}\setminus\Omega_{j-1}$,
	by \eqref{finite sum} we have 
	$$v=\eta_{j}u_{j}+(1-\eta_{j})u_{j+1}
	=\eta_{j}v_{j}+(1-\eta_{j})u_{j+1}=v_{j+1},$$
	and so in fact $v=v_{j+1}$ in $\Om_j$, for each $j\in\N$.
Thus the minimal $1$-weak upper gradient of $v$ in $\Om_{i}$ satisfies
	\[
	g_{v}=g_{v_{i+1}}
	\le \sum_{j=2}^{\infty} g_{\eta_j}|u_j-u_{j+1}|+
	\sum_{j=2}^{\infty}(\eta_{j}-\eta_{j-1}) g_{u_j}.
	\]
	Thus
	\begin{align*}
	\int_{\Om_{i}} g_v\,d\mu
	&\le \sum_{j=2}^{\infty}\int_{\Om_j} g_{\eta_j}|u_j-u_{j+1}|\,d\mu+
	\sum_{j=2}^{\infty}\int_{\Om_j} (\eta_{j}-\eta_{j-1}) g_{u_j}\,d\mu\\
	&\le \sum_{j=2}^{\infty}
	L_j\Vert u_j-u_{j+1}\Vert_{L^1(\Om_{j}\setminus \Om_{j-1})}
	+\sum_{j=2}^{\infty}\left(\int_{\Om_j}(\eta_{j}-\eta_{j-1})\,d\Vert Du\Vert+2^{-j}\eps\right)\ \ \textrm{by }\eqref{eq:choosing the energy of uis}\\
	&\le \eps/2
	+\Vert Du\Vert(\Om)+\eps/2\ \ \textrm{by }\eqref{eq:choosing the L1 closeness of uis},\eqref{eq:choosing the energy of uis}\\
	&= \Vert Du\Vert(\Om)+\eps.
	\end{align*}
	Note that $g_v$ does not depend on $i$, see
	\cite[Lemma 2.23]{BB}, and so it is well
	defined on $\Om$.
	Since $g_v$ is the minimal $1$-weak upper gradient of $v$ in
	each $\Om_i$, it is clearly
	also (the minimal) $1$-weak upper gradient of $v$ in $\Om$.
	Then by Lebesgue's monotone convergence theorem,
	\[
	\int_{\Om} g_v\,d\mu\le \Vert Du\Vert(\Om)+\eps.
	\]
\end{proof}

Theorem \ref{thm:BV theorem 1} of the introduction follows from the following theorem.

\begin{theorem}\label{thm:BV theorem 1 with more detail}
Let $u\in L^1_{\loc}(\Om)$ with $\Vert Du\Vert(\Om)<\infty$ and let $\eps>0$.
Then there exists $v\in N_{\loc}^{1,1}(\Om)\cap \liploc(\Om)$ such that
$\Vert v-u\Vert_{L^1(\Om)}<\eps$,
$\Vert v-u\Vert_{L^{s/(s-1)}(\Om)}<\eps$, $\int_{\Om}g_v\,d\mu<\Vert Du\Vert(\Om)+\eps$, and
\[
\vint{B(x,r)\cap \Om}|v-u|^{s/(s-1)}\,d\mu\to 0\quad \textrm{as }r\to 0^+
\]
uniformly for all $x\in\partial\Om$.
\end{theorem}

Note that if $u\in\BV(\Om)$ as in the formulation of Theorem \ref{thm:BV theorem 1},
then $v\in L^1(\Om)$ and so $v\in N^{1,1}(\Om)$.

\begin{proof}
Fix $x_0\in X$.
Define $\Om_0:=\emptyset$ and pick numbers $d_j\in (2^{-j},2^{-j+1})$,
$j\in\N$, such that the sets
\[
\Om_{j}:=\{x\in\Om:\,d(x,X\setminus \Om)>d_j\}\cap B(x_0,d_j^{-1})
\]
satisfy $\Vert Du\Vert(\partial \Om_j)=0$.
For each $j\in\N$, take $\eta_j\in \Lip_c(\Om_{j})$ such that $0\le \eta_j\le 1$
on $X$ and $\eta_j=1$ in $\Om_{j-1}$, and $\eta_1\equiv0$.
Note that for a fixed $r>0$, the function
\[
x\mapsto \mu(B(x,r)\cap \Om),\quad x\in\partial\Om,
\]
is lower semicontinuous and strictly positive.
Since $\partial\Om\cap \overline{B}(x_0,d_j^{-1})$ is compact
for every $j\in\N$, the numbers
\[
\beta_j:=\inf\{\mu(B(x,2^{-j})\cap \Om):
\,x\in\partial\Om\cap \overline{B}(x_0,d_{j+2}^{-1})\},\quad j\in\N,
\]
are strictly positive.
Set
\[
\delta_j:=2^{-j}\min\left\{\eps,\beta_j^{s/(s-1)}\right\}.
\]
By Lemma \ref{lem:approximating liploc functions in L with Sobolev exponent}
we find functions
$(u_i)\subset \liploc(\Om)$ such that $u_i\to u$ in $L^{s/(s-1)}_{\loc}(\Om)$
	and
	\[
	\lim_{i\to\infty}\int_\Om g_{u_i}\,d\mu=\Vert Du\Vert(\Om).
	\]
Then also $u_i\to u$ in $L^{s/(s-1)}(\Om_j)$ for every $j\in\N$, and
by Lemma \ref{lem:weak* convergence} and the fact that $\Vert Du\Vert(\partial\Om_j)=0$
we get
\[
\lim_{i\to\infty}\int_{\Om_j} g_{u_i}\,d\mu=\Vert Du\Vert(\Om_j).
\]
Then apply Lemma \ref{lem:choosing approximating function}
to obtain a function $v\in \liploc(\Om)$.
By the lemma, we have $\int_\Om g_v\,d\mu<\Vert Du\Vert(\Om)+\eps$
as desired,
and from the condition
\[
	\max\{\Vert v-u\Vert_{L^{1}(\Om_{j}\setminus \Om_{j-1})},
	\Vert v-u\Vert_{L^{s/(s-1)}(\Om_{j}\setminus \Om_{j-1})}\}<\delta_j
	\le 2^{-j}\eps
	\quad\textrm{for all }j\in\N
\]
we easily get $\Vert v-u\Vert_{L^{1}(\Om)}<\eps$ and
$\Vert v-u\Vert_{L^{s/(s-1)}(\Om)}<\eps$.
In particular, $v\in N_{\loc}^{1,1}(\Om)$ as desired.

Fix $x\in\partial\Om$. Choose the smallest $l\in\N$ such that
$x\in B(x_0,d_{l+2}^{-1})$.
Note that then $B(x,1)\cap B(x_0,d_{l-1}^{-1})=\emptyset$ (if $l\ge 2$)
and so for any $k\in\N$,
\[
B(x,2^{-k+1})\cap \Om
=B(x,2^{-k+1})\cap\Bigg(\bigcup_{j=\max\{k,l\}}^{\infty}(\Om_j\setminus \Om_{j-1})\Bigg).
\]
Now
\begin{align*}
&\frac{1}{\mu(B(x,2^{-k})\cap \Om)}\int_{B(x,2^{-k+1})\cap \Om}|v-u|^{s/(s-1)}\,d\mu\\
&\qquad\qquad=\frac{1}{\mu(B(x,2^{-k})\cap \Om)}
\sum_{j=\max\{k,l\}}^{\infty}\int_{B(x,2^{-k+1})
\cap \Om_{j}\setminus \Om_{j-1}}|v-u|^{s/(s-1)}\,d\mu\\
&\qquad\qquad\le\frac{1}{\mu(B(x,2^{-k})\cap \Om)}\sum_{j=\max\{k,l\}}^{\infty}\int_{
\Om_{j}\setminus \Om_{j-1}}|v-u|^{s/(s-1)}\,d\mu\\
&\qquad\qquad\le\frac{1}{\mu(B(x,2^{-k})\cap \Om)}\sum_{j=\max\{k,l\}}^{\infty}\delta_j^{(s-1)/s}\\
&\qquad\qquad\le\sum_{j=\max\{k,l\}}^{\infty}\frac{2^{-j}\beta_j}
{\mu(B(x,2^{-j})\cap \Om)}\\
&\qquad\qquad\le \sum_{j=\max\{k,l\}}^{\infty}2^{-j}\le 2^{-k+1}.
\end{align*}
Now it clearly follows that
\[
\vint{B(x,r)\cap \Om}|v-u|^{s/(s-1)}\,d\mu\to 0\quad \textrm{as }r\to 0^+
\]
uniformly for all $x\in\partial\Om$.
\end{proof}

We have the following approximation result for BV functions in the $L^q$-norm.

\begin{theorem}\label{thm:BV Lp approximation result}
Let $u\in L^1_{\loc}(\Om)$ with $\Vert Du\Vert(\Om)<\infty$ and let $1\le q<\infty$.
Then there exists a sequence $(u_i)\subset N^{1,1}_{\loc}(\Om)$ such that
$u_i-u \to 0$ in $L^1(\Om)\cap L^q(\Om)$ and
\[
\int_{\Om}g_{u_i}\,d\mu\to \Vert Du\Vert(\Om).
\]
\end{theorem}
\begin{proof}
For each $k=0,1,\ldots$ define
the truncation of $u$ at levels $k$ and $k+1$
\[
u_k:=\min\{1,(u-k)_+\}.
\]
Then $u_k\in L_{\loc}^1(\Om)\cap L^{\infty}(\Om)$ for each $k=0,1,\ldots$ and
$u_+=\sum_{k=0}^{\infty}u_k$.
Also note that by the coarea formula \eqref{eq:coarea},
\[
\Vert Du_k\Vert(\Om)=\int_{-\infty}^{\infty}P(\{u_k>t\},\Om)\,dt
=\int_{k}^{k+1}P(\{u>t\},\Om)\,dt.
\]
For each $k=0,1,\ldots$, from the definition of the total variation we get a
sequence $(v_{k,i})\subset N_{\loc}^{1,1}(\Om)$
with $v_{k,i}\to u_k$ in $L_{\loc}^1(\Om)$ and
\[
\int_{\Om}g_{v_{k,i}}\,d\mu\to \Vert Du_k\Vert(\Om)\quad \textrm{as }i\to\infty.
\]
In the proof of Theorem \ref{thm:BV theorem 1 with more detail} we saw
that in fact we can get $v_{k,i}- u_k\to 0$ in $L^1(\Om)$. 
Since $0\le u_k\le 1$, by truncation we can assume that also
$0\le v_{k,i}\le 1$. Then also $v_{k,i}- u_k\to 0$ in $L^q(\Om)$.
Let $\eps>0$.
For a suitable choice of indices $i=i(k)$, for $v_k:=v_{k,i(k)}$
we have
$\Vert v_{k}-u_k\Vert_{L^1(\Om)}<2^{-k-2}\eps$,
$\Vert v_{k}-u_k\Vert_{L^q(\Om)}<2^{-k-2}\eps$, and
\[
\int_{\Om}g_{v_k}\,d\mu<\Vert Du_k\Vert(\Om)+2^{-k-1}\eps
=\int_{k}^{k+1}P(\{u>t\},\Om)\,dt+2^{-k-1}\eps.
\]
Then for $v:=\sum_{k=0}^{\infty}v_k$ we have
$\Vert v-u_+\Vert_{L^1(\Om)}<\eps/2$ and $\Vert v-u_+\Vert_{L^q(\Om)}<\eps/2$.
Moreover, using e.g. \cite[Lemma 1.52]{BB} we get $g_v\le \sum_{k=0}^{\infty}g_{v_k}$
and then
\begin{align*}
\int_{\Om}g_{v}\,d\mu
\le \sum_{k=0}^{\infty}\int_{\Om}g_{v_k}\,d\mu
&\le \sum_{k=0}^{\infty}\left(\int_{k}^{k+1}P(\{u>t\},\Om)\,dt+2^{-k-1}\eps\right)\\
&=\int_{0}^{\infty}P(\{u>t\},\Om)\,dt+\eps/2\\
&=\Vert Du_+\Vert(\Om)+\eps/2
\end{align*}
again by the coarea formula.
Similarly we find a function $w\in N_{\loc}^{1,1}(\Om)$ with
$\Vert w-u_-\Vert_{L^1(\Om)}<\eps/2$, $\Vert w-u_-\Vert_{L^q(\Om)}<\eps/2$, and
$\int_{\Om}g_{w}\,d\mu<\Vert Du_-\Vert(\Om)+\eps/2$.
Then for $h:=v-w$ we have $\Vert h-u\Vert_{L^1(\Om)}<\eps$,
$\Vert h-u\Vert_{L^q(\Om)}<\eps$, and
\[
\int_{\Om}g_{h}\,d\mu<\Vert Du_+\Vert(\Om)+\eps/2+\Vert Du_-\Vert(\Om)+
\eps/2= \Vert Du\Vert(\Om)+\eps
\]
using the coarea formula once more.
In this way we get the desired sequence.
\end{proof}

Theorem \ref{thm:BV theorem 2} of the introduction follows from the following theorem.
In Example \ref{ex:counterexample for L q} we will show that here we cannot take
$u$ to be continuous or even locally bounded in $\Om$.

\begin{theorem}\label{BV thm 2 with more detail}
Let $u\in L^1_{\loc}(\Om)$ with $\Vert Du\Vert(\Om)<\infty$,
let $1\le q<\infty$, and let $\eps>0$.
Then there
exists $v\in N_{\loc}^{1,1}(\Om)$ such that
$\Vert v-u\Vert_{L^1(\Om)}<\eps$,
$\Vert v-u\Vert_{L^q(\Om)}<\eps$, $\int_{\Om}g_v\,d\mu<\Vert Du\Vert(\Om)+\eps$,
and
\[
\vint{B(x,r)\cap \Om}|v-u|^{q}\,d\mu\to 0\quad \textrm{as }r\to 0^+
\]
uniformly for all $x\in\partial\Om$.
\end{theorem}
\begin{proof}
The proof is essentially the same as for
Theorem \ref{thm:BV theorem 1 with more detail}; the difference is that
here we apply Theorem \ref{thm:BV Lp approximation result}
to find sequences $(u_{j,i})_i\subset N^{1,1}(\Om_j)$, $j\in\N$,
such that $\Vert u_{j,i}-u\Vert_{L^q(\Om_j)}\to 0$ and
$\lim_{i\to\infty}\int_{\Om_j} g_{u_{j,i}}\,d\mu=\Vert Du\Vert(\Om_j)$
as $i\to\infty$.
\end{proof}

We say that $w\in\SBV(\Om)$ if $w\in\BV(\Om)$ and $\Vert Dw\Vert^c(\Om)=0$
(recall the decomposition \eqref{eq:variation measure decomposition}).
Recall also that the jump set $S_u$ is the set of points $x\in\Om$
for which $u^{\wedge}(x)<u^{\vee}(x)$.
Denote $\Om(r):=\{x\in \Om:\,\dist(x,X\setminus \Om)>r\}$.
We have the following approximation result
for $\BV$ functions by $\SBV$ functions.
\begin{theorem}\label{thm:approximation of BV by SBV}
Let $u\in\BV(\Om)$ and let $\eps>0$.
Then there exists $w\in\SBV(\Om)$ such that
$\Vert w-u\Vert_{L^1(\Om)}<\eps$,
$\Vert w-u\Vert_{L^{\infty}(\Om)}<\eps$,
$\Vert Dw\Vert(\Om)<\Vert Du\Vert(\Om)+\eps$,
$\mathcal H(S_w\setminus S_u)=0$,
and
\[
\lim_{r\to 0^+}\Vert w-u\Vert_{L^{\infty}(\Om\setminus \Om(r))}=0.
\]
\end{theorem}
\begin{proof}
This is given in \cite[Corollary 5.15]{L-Appr}; for the above limit
see \cite[Eq. (3.7), (3.10)]{L-Appr}.
\end{proof}

The following approximation result for $\BV$ functions
by means of functions with a jump set of finite Hausdorff measure is given
as part of \cite[Theorem 5.3]{L-2Appr}.

\begin{theorem}\label{thm:approximation}
Let $u\in\BV(\Om)$ and let $\eps,\delta>0$. Then
we find $w\in \BV(\Om)$ such that $\Vert w-u\Vert_{L^1(\Om)}<\eps$,
\[
\Vert D(w-u)\Vert(\Om)
\le 2\Vert Du\Vert(\{0<u^{\vee}-u^{\wedge}< \delta\})+\eps,
\]
$\Vert w-u\Vert_{L^{\infty}(\Om)}\le 10\delta$, and
$\mathcal H(S_{w}\setminus \{u^{\vee}-u^{\wedge}\ge \delta\})=0$.
\end{theorem}

We apply this theorem first to obtain the following proposition.

\begin{proposition}\label{prop:boundary values of SBVf}
Let $u\in\BV(\Om)$ and let $\eps>0$. Then
we find $v\in \BV(\Om)$ such that $\Vert v-u\Vert_{\BV(\Om)}<\eps$,
$\Vert v-u\Vert_{L^{\infty}(\Om)}<\eps$, $\mathcal H(S_v)<\infty$, and
\[
\lim_{r\to 0^+}\Vert v-u\Vert_{L^{\infty}(\Om\setminus \Om(r))}=0.
\]
\end{proposition}

\begin{proof}
Take numbers $\delta_j\searrow 0$, $0<\delta_j<\eps/20$, such that
\begin{equation}\label{eq:sum of delta j jump sets}
\sum_{j=2}^{\infty}\Vert Du\Vert(\{0<u^{\vee}-u^{\wedge}< \delta_j\})
<\frac{\eps}{4}.
\end{equation}
Note that by the decomposition \eqref{eq:variation measure decomposition},
$\mathcal H(\{u^{\vee}-u^{\wedge}>t\})<\infty$ for all $t>0$.
Thus we can take a strictly decreasing sequence of numbers $r_j>0$ so
that the sets $\Om_j:=\Om(r_j)$ satisfy
(let $\Om_0:=\emptyset$)
\[
\mathcal H((\Om_j\setminus \Om_{j-2})
\cap \{u^{\vee}-u^{\wedge}\ge\delta_j\})<2^{-j}\eps
\]
for all $j=3,4,\ldots$.
Then
\begin{equation}\label{eq:choice of set Omega j}
\sum_{j=2}^{\infty}\mathcal H((\Om_{j}\setminus \Om_{j-2})
\cap \{u^{\vee}-u^{\wedge}\ge\delta_j\})
<\mathcal H(\{u^{\vee}-u^{\wedge}\ge\delta_2\})+\eps.
\end{equation}
Also choose functions
$\eta_j\in \Lip(X)$ supported in $\Om_j$, $j\in\N$,
such that $0\le \eta_j\le 1$
on $X$ and $\eta_j=1$ in $\Om_{j-1}$, with $\eta_1\equiv 0$.
For each $j\in\N$, apply Theorem \ref{thm:approximation} to find a function
$v_j\in \BV(\Om)$
satisfying
\begin{equation}\label{eq:vj u and test functions}
\max\{\Vert{g_{\eta_j}+g_{\eta_{j-1}}\Vert_ {L^{\infty}(\Om)}},1\}
\cdot \Vert v_j-u\Vert_{L^1(\Om)}<2^{-j-1}\eps
\end{equation}
as well as
\begin{equation}\label{eq:BV distance of vj and u}
\Vert D(v_j-u)\Vert(\Om)
\le 2\Vert Du\Vert(\{0<u^{\vee}-u^{\wedge}< \delta_j\})+2^{-j-1}\eps,
\end{equation}
$\Vert v_j-u\Vert_{L^{\infty}(\Om)}\le 10\delta_j$, and
$\mathcal H(S_{v_j}\setminus \{u^{\vee}-u^{\wedge}\ge \delta_j\})=0$.
Let
	\begin{equation}\label{eq:definition of w by means of etas}
	v:=\sum_{j=2}^{\infty}(\eta_{j}-\eta_{j-1})v_{j}.
	\end{equation}
Then
	\begin{align*}
	\Vert v-u\Vert_{L^1(\Om)}
	=\Vert \sum_{j=2}^{\infty}(\eta_j-\eta_{j-1})(v_{j}-u)\Vert_{L^1(\Om)}
	\le \sum_{j=2}^{\infty} \Vert v_j-u\Vert_{L^{1}(\Om)}
	\le \sum_{j=2}^{\infty} 2^{-j-1}\eps
	=\eps/4.
	\end{align*}
	Since $\Vert v_j-u\Vert_{L^{\infty}(\Om)}\le 10\delta_j<\eps/2$,
	also $\Vert v-u\Vert_{L^{\infty}(\Om)}<\eps$.
	It is also easy to check that
	$\lim_{r\to 0^+}\Vert v-u\Vert_{L^{\infty}(\Om\setminus \Om(r))}=0$.
	
	Clearly $\sum_{j=2}^{k}(\eta_{j}-\eta_{j-1})(v_{j}-u)\to v-u$ in $L^1_{\loc}(\Om)$
	as $k\to\infty$. Thus by lower semicontinuity and a Leibniz rule
	(see \cite[Lemma 3.2]{HKLS}),
	\begin{align*}
	&\Vert D(v-u)\Vert(\Om)
	\le \liminf_{k\to\infty}\bigg\Vert D\sum_{j=2}^{k}(\eta_{j}-\eta_{j-1})(v_{j}-u)
	\bigg\Vert(\Om)\\
	&\qquad\qquad\le \sum_{j=2}^{\infty}\Vert D((\eta_{j}-\eta_{j-1})(v_{j}-u)) \Vert(\Om)\quad\textrm{by }\eqref{eq:BV functions form vector space}\\
	&\qquad\qquad\le \sum_{j=2}^{\infty}\left(\Vert D(v_{j}-u) \Vert(\Om)
	+\int_{\Om}(g_{\eta_j}+g_{\eta_{j-1}})|v_j-u|\,d\mu\right)\\
	&\qquad\qquad<\sum_{j=2}^{\infty}\left(2\Vert Du\Vert(\{0<u^{\vee}-u^{\wedge}< \delta_j\})+2^{-j-1}\eps\right)+\sum_{j=2}^{\infty}2^{-j-1}\eps\quad\textrm{by }\eqref{eq:BV distance of vj and u},\eqref{eq:vj u and test functions}\\
	&\qquad\qquad< \eps/2+\eps/4+\eps/4\quad\textrm{by }\eqref{eq:sum of delta j jump sets}\\
	&\qquad\qquad=\eps.
	\end{align*}
Finally we want to show that $\mathcal{H}(S_v)<\infty$.
Note that \eqref{eq:definition of w by means of etas}
is a locally finite sum.
If $x\in S_{(\eta_{j}-\eta_{j-1})v_{j}}$, then we get $x\in S_{v_j}$, and so
$S_v\subset \bigcup_{j=2}^{\infty}\left(S_{v_j}\cap (\Om_{j}\setminus \Om_{j-2})\right)$.
By the fact that
$\mathcal H(S_{v_j}\setminus \{u^{\vee}-u^{\wedge}\ge \delta_j\})=0$
for all $j\in\N$ and by \eqref{eq:choice of set Omega j},
we find that
\begin{align*}
	\mathcal H(S_v)&\le \sum_{j=2}^{\infty}\mathcal{H}(S_{v_j}\cap(\Omega_j\setminus\Omega_{j-2}))
	\le \sum_{j=2}^{\infty}\mathcal{H}(\{u^{\vee}-u^{\wedge}\ge \delta_j\}\cap(\Omega_j\setminus\Omega_{j-2}))
	\\&<\mathcal H(\{u^{\vee}-u^{\wedge}\ge\delta_2\})+\eps<\infty,
	\end{align*}
as desired.
\end{proof}

Now we can prove Theorem \ref{thm:BV theorem 3} of the introduction.
In Example \ref{ex:counterexample for L infinity}
we will show that here we cannot have $\mathcal H(S_v)=0$.

\begin{proof}[Proof of Theorem \ref{thm:BV theorem 3}]
First apply Proposition \ref{prop:boundary values of SBVf} to
find $\widehat{w}\in \BV(\Om)$ such that $\Vert \widehat{w}-u\Vert_{\BV(\Om)}<\eps/4$,
$\mathcal H(S_{\widehat{w}})<\infty$, and
\[
\lim_{r\to 0^+}\Vert \widehat{w}-u\Vert_{L^{\infty}(\Om\setminus \Om(r))}=0.
\]
Then apply Theorem \ref{thm:approximation of BV by SBV} to find
$w\in\SBV(\Om)$ such that
$\Vert w-\widehat{w}\Vert_{L^1(\Om)}<\eps/4$,
$\Vert Dw\Vert(\Om)<\Vert D\widehat{w}\Vert(\Om)+\eps/4$,
$\mathcal H(S_w\setminus S_{\widehat{w}})=0$,
and
\[
\lim_{r\to 0^+}\Vert w-\widehat{w}\Vert_{L^{\infty}(\Om\setminus \Om(r))}=0.
\]
In total, we have $w\in\SBV(\Om)$ such that
$\Vert w-u\Vert_{L^1(\Om)}<\eps/2$,
$\Vert Dw\Vert(\Om)<\Vert Du\Vert(\Om)+\eps/2$,
$\mathcal H(S_w)<\infty$,
and
\[
\lim_{r\to 0^+}\Vert w-u\Vert_{L^{\infty}(\Om\setminus \Om(r))}=0.
\]
Take $\Om'\Subset \Om$ such that $\Vert Dw\Vert(\Om\setminus \Om')<\eps/2$
and $\mathcal H(S_w\setminus \Om')<\eps$,
and take a function $\eta\in \Lip_c(\Om)$ with $0\le \eta \le 1$ on $X$ and
$\eta=1$ in $\Om'$.
From the definition of the total variation,
take a sequence $(w_i)\subset\liploc(\Om)$ such that
$w_i\to w$ in $L^1_{\loc}(\Om)$ and
$\lim_{i\to \infty}\Vert Dw_i\Vert(\Om)=\Vert Dw\Vert(\Om)$. Define for each $i\in\N$
\[
v_i:=\eta w_i+(1-\eta)w.
\]
Then clearly $\lim_{i\to\infty}\Vert v_i-w\Vert_{L^1(\Om)}=0$ and by a Leibniz rule
(see \cite[Lemma 3.2]{HKLS}) and since $g_\eta$ is bounded,
\begin{align*}
\Vert Dv_i\Vert(\Om)
&\le \int_{\Om}|w_i-w|g_{\eta}\,d\mu+\Vert Dw_i\Vert(\Om)
+\Vert Dw\Vert(\Om\setminus \Om')\\
&\to \Vert Dw\Vert(\Om)+\Vert Dw\Vert(\Om\setminus \Om')< \Vert Du\Vert(\Om)+\eps.
\end{align*}
Thus if we choose $v:=v_i$ for suitably large $i\in\N$, we have
$\Vert v-u\Vert_{L^1(\Om)}<\eps$ and $\Vert Dv\Vert(\Om)<\Vert Du\Vert(\Om)+\eps$,
and so in particular $v\in \BV(\Om)$.
It is then easy to check that
in fact $v\in \SBV(\Om)$.
Since $S_{w_i}=\emptyset$ for all $i\in\N$, we have
$S_{v_i}\subset S_w\setminus \Om'$
for all $i\in\N$, and since
$\mathcal H(S_w\setminus \Om')<\eps$,
in fact $\mathcal H(S_{v})<\eps$.
Finally,
\[
\lim_{r\to 0^+}\Vert v-u\Vert_{L^{\infty}(\Om\setminus \Om(r))}
=\lim_{r\to 0^+}\Vert w-u\Vert_{L^{\infty}(\Om\setminus \Om(r))}=0
\]
as required.
\end{proof}

To complete this section, we give the proof of
Corollary \ref{BV-Sobolev} by using Theorem
\ref{thm:BV theorem 1 with more detail}
(or Theorem \ref{BV thm 2 with more detail}).

\begin{proof}[Proof of Corollary \ref{BV-Sobolev}]
Assume that $\mathbb Z(\bx, \mathcal H)$ is the trace space of
$\BV(\Omega)$, i.e., the trace operator $u\mapsto Tu$ in
Definition \ref{trace} is a bounded linear surjective operator
from $\BV(\Om)$ to $\mathbb Z(\bx, \mathcal H)$.
{From the definition of the total variation \eqref{eq:total variation}
we immediately get $N^{1,1}(\Om)\subset \BV(\Om)$
with $\Vert \cdot\Vert_{\BV(\Om)}\le \Vert \cdot\Vert_{N^{1,1}(\Om)}$. Thus}
the trace operator $u\mapsto Tu$
is still a bounded linear operator from $N^{1,1}(\Om)$ to
$\mathbb Z(\bx, \mathcal H)$. Hence it remains to show the
surjectivity. For any $f\in \mathbb Z(\bx, \mathcal H)$, we
know that there is a function $u\in \BV(\Omega)$ such that $Tu=f$.
It follows from Theorem \ref{thm:BV theorem 1 with more detail}
(or Theorem \ref{BV thm 2 with more detail}) that there is a function
$v\in N^{1,1}(\Omega)$ such that $Tv=Tu=f$, since
\begin{align}
\lim_{r\rightarrow 0^+}\vint{B(x, r)\cap\Om} &|v-f(x)|\, d\mu\leq \lim_{r\rightarrow 0^+} \vint{B(x, r)\cap\Om} |u-v|+|u-f(x)|\, d\mu\notag\\
&\leq \lim_{r\rightarrow 0^+} \left(\vint{B(x, r)\cap\Om} |u-v|^{s/s-1}\, d\mu\right)^{(s-1)/s} +\lim_{r\rightarrow 0^+} \vint{B(x, r)\cap\Om} |u-f(x)|\, d\mu\notag\\
&=0,\quad \text{for} \   \text{$\mathcal H$-a.e.}  \ x\in \bx.\label{limit-BV}
\end{align}
{This gives the surjectivity as desired.

Conversely,}
assume that $\mathbb Z(\bx, \mathcal H)$ is the trace
space of $N^{1,1}(\Om)$, i.e., the trace operator $u\mapsto Tu$ in
Definition \ref{trace} is a bounded linear surjective operator
from $N^{1,1}(\Om)$ to $\mathbb Z(\bx, \mathcal H)$. Then for any
$h\in \BV(\Om)$, without loss of generality, we may assume that
$\|h\|_{\BV(\Omega)}>0$. By
Theorem \ref{thm:BV theorem 1 with more detail},
{choosing $\eps=\|h\|_{\BV(\Omega)}/2$,}
there is a function $v\in N^{1,1}(\Om)$ with
$\|v\|_{N^{1,1}(\Om)}\leq 2\|h\|_{\BV(\Om)}$ and 
\[
\vint{B(x,r)\cap \Om}|v-h|^{s/(s-1)}\,d\mu\to 0\quad \textrm{as }r\to 0^+
\]
uniformly for all $x\in\partial\Om$. Then we have that $Th=Tv$ by a
similar argument to \eqref{limit-BV}, and that
$$\|Th\|_{\mathbb Z(\bx, \mathcal H)}=\|Tv\|_{\mathbb Z(\bx, \mathcal H)}\lesssim \|v\|_{N^{1,1}(\Om)} \leq 2 \|h\|_{\BV(\Om)}.$$
Hence the trace $Th$ exists for any $h\in \BV(\Om)$ and the trace
operator $h\rightarrow Th$ is linear and bounded from $\BV(\Om)$
to $\mathbb Z(\bx, \mathcal H)$. Moreover, the surjectivity of the
trace operator follows immediately from $N^{1,1}(\Om)\subset \BV(\Om)$.
Thus $\mathbb Z(\bx, \mathcal H)$ is also the trace space of $\BV(\Om)$.
\end{proof}

\begin{remark}
The trace spaces of $\BV(\Omega)$ and $N^{1,1}(\Omega)$ are also the
same with respect to any given boundary measure $\widetilde{\mathcal H}$
under Definition \ref{trace-space-measure}. 
\end{remark}

\section{Traces of $N^{1,1}(\Om)$ and $M^{1,1}(\Om)$}

In this section, let $\Om\subset X$ be an
{arbitrary nonempty open set with nonempty complement.}

We will work with Whitney coverings of open sets.
{For a ball $B=B(x,r)$ and a number $a>0$, we use the notation $a B:=B(x,ar)$.}
We can choose a Whitney covering $\{B_j=B(x_j, r_j)\}_{j=1}^{\infty}$ of $\Omega$ such that:

\begin{enumerate}
\item for each $j\in \N$,
$$r_j=\dist (x_j, X\setminus \Omega)/100\lambda,$$

\item for each $k\in \N$, the ball
{$20\lambda B_k$ meets at most
$C_0=C_0(C_d)$}
balls $20\lambda B_j$ (that is, a bounded overlap property holds),

\item if $20\lambda B_k$ meets $20\lambda B_j$, then $r_j\leq 2r_k$;
\end{enumerate}
{see e.g. \cite[Proposition 4.1.15]{HKST} and its proof.}
Given such a covering of $\Omega$, we find a partition of unity
$\{\phi_j\}_{j=1}^{\infty}$ subordinate to the covering, that is, for each
$j\in \N$ the function $\phi_j$ is $c/r_j$-Lipschitz,
{$c=c(C_d)$, with}
$\supp(\phi_j)\subset 2 B_j$ and $0\leq \phi_j\leq 1$, such that
$\sum_j \phi_j=1$ on $\Omega$;
{see e.g. \cite[p. 103]{HKST}.}
We define a {\it discrete convolution $u_W$}
of $u\in L_{\rm loc}^1(\Omega)$ with respect to the Whitney
covering by
$$u_W:=\sum_{j=1}^{\infty} u_{B_j}\phi_j.$$
In general, $u_W\in \Lip_{\rm loc}(\Omega)\subset L_{\rm loc}^1(\Omega)$.

\begin{theorem}\label{convolution}
  For any function $u\in N^{1,1}(\Omega)$, there exists a constant
  {$0<c_H=c_H(\lambda)<1$} such that the discrete convolution $u_W$ of $u$  with respect to the Whitney covering $\{B_j=B(x_j, r_j)\}_{j=1}^{\infty}$ is in $M^{1,1}_{c_H}(\Omega)$ with the norm estimate
$$\|u_W\|_{M^{1,1}_{c_H}(\Omega)}\lesssim \|u\|_{N^{1,1}(\Omega)}.$$
\end{theorem}
\begin{proof}
First we consider the $L^1$-norm of $w_W$. By the bounded overlap property
of the Whitney covering $\{B_j=B(x_j, r_j)\}_{j=1}^{\infty}$,
it follows from the facts $\supp(\phi_j)\subset 2 B_j$ and
$0\leq \phi_j\leq 1$
that
\[
\|u_W\|_{L^1(\Omega)}
\le  \sum_{j=1}^{\infty}  \mu(2B_j) \vint{B_j} |u|\, d\mu\\
\le C_d \sum_{j=1}^{\infty} \int_{B_j} |u|\, d\mu \lesssim \|u\|_{L^{1  }(\Omega)}.
\]

Next, for
{the minimal $1$-weak upper gradient}
$g_u$ of $u$, we will give an admissible function
$g$ that satisfies \eqref{H-gradient} when the pair of points
$x, y$ satisfy \eqref{H-gradient-ball} with $c_H=1/50\lambda$. We claim
that the admissible function $g$ can be defined as follows: for any point
$x\in \Omega$, we define
 \begin{equation}\label{u_W-gradient}
 g(x):=C\sum_{j=1}^{\infty}\chi_{B_j}(x) \vint{60 \lambda B_j} g_u\, d\mu
 \end{equation}
{with $C=C(C_d,C_P,\lambda)$.}
Indeed, for any pair of points $x, y\in \Omega$ satisfying \eqref{H-gradient-ball}, without loss of generality, we may assume that $\dist(x, X\setminus \Om)\leq \dist(y, X\setminus \Om)$ and $x\in B_j$, $y\in B_i$ for some $i, j\in\N$.
Recalling the properties of the Whitney covering, we have that
$$\dist(x,  X\setminus \Om)\leq \dist(x_j,  X\setminus \Om)+r_j  
= (100\lambda+1)r_j.$$
Hence we have
$$d(y, x_j)\leq d(x, y)+r_j\leq \frac{1}{50\lambda}
\dist (x, X\setminus \Omega)+r_j < 4 r_j,$$
which means $y\in 4 B_j$. Hence $20\lambda B_i\cap 20\lambda B_j\not=\emptyset$,
and so $r_i\leq 2r_j$.
 Hence $B_i\subset 10 B_j$. Moreover, if $2B_k\cap B_i\not=\emptyset$,
then $r_k\leq 2r_i$
and so $B_k\subset 6B_i\subset 20B_j$.
Recall that the function $\phi_k$ is $c/r_k$-Lipschitz for any $k\in \N$
and that $\sum_{k}\phi_k=1$ on $\Om$.
Then by  the bounded overlap property of the Whitney covering and the
Poincar\'e inequality for $u$ and $g_u$, we have that
\begin{align}
|u_W(x)-u_W(y)|
&=\left|\sum_{k=1}^{\infty} u_{B_k}\phi_k(x)-\sum_{k=1}^{\infty} u_{B_k}\phi_k(y) \right|\notag\\
&=\left|\sum_{k=1}^{\infty} (u_{B_k}-u_{B_j})\phi_k(x)-\sum_{k=1}^{\infty} (u_{B_k}-u_{B_j})\phi_k(y) \right|\notag\\
&\leq \sum_{k=1}^{\infty} |u_{B_k}-u_{B_j}||\phi_k(x)-\phi_k(y)|\notag\\
&\leq d(x, y)\sum_{\{k:\, 2B_k\cap(B_j\cup B_i) \not=\emptyset\}}  \frac{c}{r_k} |u_{B_k}-u_{B_j}|\notag\\
&\lesssim d(x, y)\frac {c}{r_j} \vint{20 B_j} |u-u_{20 B_j}|\, d\mu\label{local-poincare}\\
&\leq C d(x, y)\vint{20 \lambda B_j} g_u\, d\mu,\notag
\end{align}
where $C$ is a constant depending
on $\lambda, c, C_d, C_P$ and $C_0$ only,
and thus in fact only on $C_d,C_P,\lambda$.
Thus, the function $g$ defined in \eqref{u_W-gradient}
is an admissible function for $u_W$.

At last, we show the $L^1$-norm estimate for $g$.
It follows from the bounded overlap property of the Whitney covering that
\begin{align*}
\int_\Om g(x)\, d\mu(x)&\leq \sum_{j=1}^{\infty} \int_{B_j} g(x)\, d\mu(x)\lesssim  \sum_{j=1}^{\infty}  \mu(B_j) \vint{20 \lambda B_j} g_u\, d\mu\\
&\lesssim \sum_{j=1}^{\infty} \int_{20\lambda B_j} g_u(x)\, d\mu(x) \lesssim \int_\Om g_u(x)\, d\mu(x)= \|g_u\|_{L^{1  }(\Omega)}.
\end{align*} 
\end{proof}

Recall the homogeneous dimension $s>1$ from \eqref{eq:homogeneous dimension}.

\begin{theorem}[{\cite[Theorem 9.2]{H03}}]\label{thm:PI}
Let $\sigma>1$ and let $B=B(x,r)$ be a ball in $X$.
If $u\in M^{1, p}(\sigma B, d, \mu)$ and $g$ is an admissible function
in \eqref{H-gradient}, where $p\geq s/(s+1)$, then
\begin{equation}
\vint{B} |u-u_B|\, d\mu \leq C r\left(\vint{\sigma B} g^p\, d\mu\right)^{1/p},
\end{equation} 
with $C$ depending on $C_d$, $p$, and $\sigma$ only.
\end{theorem}

Next we will consider the relationship between $M^{1,1}_{c_H}(\Om)$ and
$M^{1,1}(\Om)$. The next theorem shows that when $\Omega\subset X$ is a
uniform domain, $M^{1,1}_{c_H}(\Omega)$ and $M^{1,1}(\Omega)$ are the same.
The case $X=\R^n$, i.e. the Euclidean case was proved in \cite[Theorem 19]{KS08}.
Before stating the theorem, we first give the definition of uniform domain.

\begin{definition}
A domain $\Omega\subset X$ is called  {\it uniform} if there is a
constant $c_U\in(0, 1]$ such that every pair of distinct points
$x, y\in\Omega$ can be connected by a curve
$\gamma\colon [0, \ell_\gamma]\to\Omega$ parametrized by arc-length
such that $\gamma(0)=x$, $\gamma(\ell_\gamma)=y$,
$\ell_\gamma\leq c_U^{-1}d(x, y)$, and 
\begin{equation}
\label{Uniform}\dist(\gamma(t), X\setminus\Omega)\geq c_U \min\{t, \ell_\gamma-t\}
\quad {\rm for\ all}\  \ t\in[0, \ell_\gamma].
\end{equation}
\end{definition}

\begin{theorem}\label{local-global}
Assume $\Omega\subset X$ is a uniform domain. Then for any $0<c_H<1$, we have $M^{1,1}_{c_H}(\Omega)=M^{1,1}(\Omega)$ with equivalent norms.
\end{theorem}
\begin{proof}
Choose arbitrary $x, y\in \Omega$.
By modifying the standard covering argument in uniform domains
(see \cite{HK98,HO97,KS08} for details),
from the uniformity condition we deduce 
easily that there is a chain of balls $B_k$ resembling a cigar that joins the
points $x$ and $y$.
More precisely, there are balls $B_k:= B(z_k, r_k)$
with $k\in\Z$ and $z_k\in \Omega$ such that for each $k$ one has
for some $c'=c'(\lambda,c_H,c_U)$
$$15\lambda/c_H B_k\subset \Omega\quad\text{and}\quad
r_k\geq \frac{1}{c'} \min\{d(z_k, x), d(z_k, y)\},$$
with also $B_k\cap B_{k+1}\not=\emptyset$, and $r_k/2\leq r_{k+1}\leq 2r_k$.
In addition, $\lim_{k\rarrow +\infty}d(x, B_k)=0=\lim_{k\rarrow -\infty}d(y, B_k)$.
Finally, we may assume that $\sum_{k\in \Z} r_k \leq C' d(x, y)$.

{Let $u\in M^{1,1}_{c_H}(\Omega)$ with admissible function $g\in L^1(\Om)$.
We can zero extend $g$ outside $\Om$.}
Since $15\lambda/c_H B_k \subset \Omega$ and $c_H<1$, then for any $x_0, y_0\in 5\lambda B_k$,
we have
$$d(x_0,y_0)\leq 10 \lambda r_k\leq c_H(15\lambda/c_H-5\lambda)r_k \leq c_H\min\{\dist(x_0, X\setminus\Omega), \dist(y_0, X\setminus \Om)\}.$$
Hence, for any $x_0, y_0\in 5\lambda B_k$, the condition \eqref{H-gradient-ball}
is satisfied. Thus, $u\in M^{1,1}(5\lambda B_k)$ for any $k\in \Z$. It follows
from the Poincar\'e inequality
in Theorem \ref{thm:PI} on the ball $5 B_k$ with $\sigma=\lambda$ that 
\begin{align*}
|u_{B_k}-u_{B_{k+1}}| &\lesssim \vint{5B_k} |u-u_{5B_k}| \lesssim r_k\left(\vint{5\lambda B_k} g^{s/(s+1)} \, d\mu\right)^{(s+1)/s}\\
&\lesssim r_k\left(\vint{(5\lambda+2 c')B_k} g^{s/(s+1)} \, d\mu\right)^{(s+1)/s}\\
& \lesssim {r_k}\left(\left(\mathcal M g^{s/(s+1)}(x)\right)^{(s+1)/s}+ \left(\mathcal M g^{s/(s+1)}(y)\right)^{(s+1)/s}\right),
\end{align*}
where $s$ is the associated homogeneous dimension.
 {Here the last inequality follows from the fact that either $x$ or $y$ is contained in $2c'B_k\subset (5\lambda+2 c')B_k$}.

If $x, y$ are Lebesgue points of $u$, we have
$|u(x)-u(y)|\leq \sum_{k\in\Z}|u_{B_k}-u_{B_{k+1}}|$.
By summing over $k$, it follows that
%
$$|u(x)-u(y)|\leq d(x, y)(\tilde g(x)+\tilde g(y)),$$
{where
	$\widetilde g(x)=2C{\left(\mathcal M
		g^{s/(s+1)}(x)\right)^{(s+1)/s}}$.}
The conclusion follows from the Hardy-Littlewood maximal inequality.
\end{proof}

\begin{remark}\label{M_1}
From the proof of Theorem \ref{local-global},
we know that if $X$ is a geodesic space, i.e., for any $x, y\in X$,
there exists a curve $\gamma$ in $X$
such that $\ell_{\gamma}=d(x, y)$, then
$M^{1,1}_{c_1}(\Om)=M^{1,1}_{c_2}(\Om)$ with equivalent norms for any two
constants $0<c_1, c_2<1$. This fact coincides with the case
$\Omega\subset \R^n$, where $\R^n$ is a geodesic space.
When $\Omega\subset \R^n$,
for any $0<c_H<1$, we obtain $M^{1,1}_{c_H}(\Om)=M^{1,1}_{ball}(\Om)$. Here
we refer to \cite{KS08,Z11} for more details about the space $M^{1,1}_{ball}(\Om)$.
\end{remark}
   
 To ``achieve'' the boundary values, we need the following proposition.

 \begin{proposition}[{\cite[Proposition 6.5]{LS18}}]\label{trace-H}
Let $u\in \BV(\Omega)$.
Then the discrete convolution $u_W$ of $u$ satisfies
$$\lim_{r\rarrow 0^+} \frac{1}{\mu(B(x, r))} \int_{B(x, r)\cap \Omega} |u_W-u|\, d\mu=0$$
for $\mathcal{H}$-a.e. $x\in \partial \Omega$.
\end{proposition}

The above proposition considers the measure $\mathcal H$ on $\partial \Om$, that is,
the codimension $1$ Hausdorff measure.  But this may not be the measure we
really want to study. For example, a classical problem is to study the trace
spaces of weighted Sobolev spaces on Euclidean spaces. For the half plane
$\Omega=\R^{2}_+:=\{x=(x_1, x_2)\in \R^2:\, x_2>0 \}$ and the measure
$d\mu(x)=w_\lambda(x)\, dm_2(x)$ with $m_2$ the $2$-dimensional Lebesgue
measure and $w_\lambda(x):=\log^\lambda\left(\max\{e, e/|x_2|\}\right)$,
$\lambda>0$,  the codimension $1$ Hausdorff measure on $\partial \R^2_+=\R$
is not even $\sigma$-finite and hence is not the $1$-dimensional Lebesgue measure
that we usually study, see Example \ref{ex:weight}. Thus, it is reasonable
to consider the equivalence of the traces of $N^{1,1}(\Om)$ and $M^{1,1}(\Omega)$
under any general boundary measure $\widetilde {\mathcal H}$ on $\partial \Omega$.
Thus, we introduce the following lemma.

\begin{lemma}\label{trace-h}
Assume $\Omega$ satisfies a measure doubling condition \eqref{measure-doubling},
i.e., $\mu_{|\Omega}$ is doubling.
Let $u\in L_{\rm loc}^1(\Omega)$ and $z\in \partial \Om$.
Assume that there is $a\in \R$ such that 
$$\lim_{r\rarrow 0^+}\vint{B(z, r)\cap \Omega} |u-a|\, d\mu=0.$$
Then the discrete convolution $u_W$ of $u$ satisfies
$$\lim_{r\rarrow 0^+}\vint{B(z, r)\cap \Omega} |u_W-a|\, d\mu=0.$$
\end{lemma}

\begin{proof}
In the Whitney covering $\{B_k\}_{k=1}^{\infty}$,
recall that for any $B_k=B(x_k, r_k)$ we have $r_k=\dist(x_k, X\setminus\Omega)/100\lambda$. If $2B_k\cap B(z, r)\not=\emptyset$, then
$$2r_k+ r\geq d(x_k, z)\geq \dist(x_k, X\setminus \Omega)=100\lambda r_k,$$
which implies 
$$\bigcup_{\{k\in \N:\, 2B_k\cap B(z, r)\not=\emptyset\}} B_k\subset B(z, 2r).$$
Then we have 
\begin{align*}
\int_{B(z, r)\cap \Omega} |u_W-a|\, d\mu
&= \int_{B(z, r)\cap \Omega} \left|\sum_{k=1}^{\infty}(\phi_k u_{B_k}-\phi_k a)\right|\, d\mu\\
&\le \int_{B(z, r)\cap \Omega} \sum_{k=1}^{\infty}|\phi_k| |u_{B_k}- a|\, d\mu\\
&\le \int_{B(z, r)\cap \Omega} \sum_{k=1}^{\infty}\chi_{2B_k} |u_{B_k}- a|\, d\mu\\
&\le \int_{B(z, r)\cap \Omega} \sum_{k=1}^{\infty}\chi_{2B_k} \vint{B_k}|u-a|\,d\mu\, d\mu\\
&\le C_d\sum_{\{k\in \N:\, 2B_k\cap B(z, r)\not=\emptyset\}} \int_{B_k}|u-a|\,d\mu\\
&\lesssim   \int_{B(z,2r)\cap \Om}|u-a|\,d\mu
\end{align*}
by the bounded overlap property.
Thus, the doubling property of $\mu_{|\Omega}$ gives the estimate
\begin{align*}
\vint{B(z, r)\cap \Omega} |u_W-a|\, d\mu
&\lesssim \vint{B(z, 2r)\cap \Omega} |u-a|\, d\mu.
\end{align*}
The result follows by passing to the limit.
\end{proof}

\begin{proof}[Proof of Theorem \ref{Sobolev-H}, Theorem \ref{Sobolev-H-measure}, and Corollary \ref{SObolev-M-2}] 
Theorem \ref{Sobolev-H} follows immediately by combining Theorem \ref{convolution}, Theorem \ref{local-global} and Proposition \ref{trace-H}, while Theorem \ref{Sobolev-H-measure} follows immediately by combining Theorem \ref{convolution}, Theorem \ref{local-global} and Lemma \ref{trace-h}.

For Corollary \ref{SObolev-M-2}, by adapting the proof of Corollary \ref{BV-Sobolev}, we obtain the result using Theorem \ref{Sobolev-H-measure}.
Note that $M^{1,1}(\Om)\subset N^{1,1}(\Om)$ with
$\Vert \cdot\Vert_{N^{1,1}(\Om)}\lesssim \Vert \cdot\Vert_{M^{1,1}(\Om)}$,
see \cite[Theorem 8.6]{H03}.
\end{proof}

\section{Examples and applications}\label{application-example}

The following example shows that in Theorem \ref{thm:BV theorem 2} we cannot
take a function $v\in \liploc(\Om)$, or even $v\in L^{\infty}_{\loc}(\Om)$.

\begin{example}\label{ex:counterexample for L q}
Let $X=\R^2$ (unweighted) and let $\Om:=B(0,1)$.
We find  a sequence $\{x_k\}$ that is dense in $ B(0,1)$.
Take
\[
u_k(x):=|x-x_k|^{-1+1/k},\quad k\in\N.
\]
Then $\Vert u_k\Vert_{L^1(\Om)}<\infty$
and the minimal $1$-weak upper gradient satisfies
(see \cite[Proposition A.3]{BB})
\[
g_{u_k}(x)=|\nabla u_k(x)|=(-1+1/k)|x-x_k|^{-2+1/k}
\]
and so
$$\int_{B(0,1)}g_{u_k}\,dx\lesssim\int_{B(0,1)}|x-x_k|^{-2+1/k}\,dx
\le \int_{B(0,2)}|x|^{-2+1/k}\,dx<\infty.$$
Let $$u(x):=\sum_{k}2^{-k}\frac{u_k}{\Vert u_k\Vert_{N^{1,1}(B(0,1))}}.$$
 Then using e.g. \cite[Lemma 1.52]{BB} we see that $u$
has a $1$-weak upper gradient
$$\sum_k 2^{-k}\frac{g_{u_k}}{\Vert u_k\Vert_{N^{1,1}(B(0,1))}},$$
which implies $u\in N^{1,1}(B(0,1))$. We know that the homogeneous
dimension $s$ of $\R^2$ is 2, and then $\frac{s}{s-1}=2$.
On the other hand, we can see that for any $q>2$,
we have for all sufficiently large $k\in\N$
$$\int_{B(x_k,r)\cap B(0,1)}|u_k|^q\,dx=\infty\quad\textrm{for all }r>0,$$
and then for all balls $B\cap B(0,1)\neq \emptyset$  we have
$\int_{B\cap B(0,1)}|u|^q\,dx=\infty$. Given
$v\in \liploc(B(0,1))$, we know that
$v\in L^q_{\loc}(B(0,1))$.
Therefore we have $\Vert v-u\Vert_{L^q(B\cap B(0,1))}=\infty$ for all balls
$B\cap B(0,1)\neq \emptyset$, which contradicts the desired
conclusion in Theorem \ref{thm:BV theorem 2}.
\end{example}

The following example shows that in Theorem \ref{thm:BV theorem 3} we cannot
take a function $v$ with $\mathcal H(S_v)=0$.

\begin{example}\label{ex:counterexample for L infinity}
Let $X=\R^2$ (unweighted) and let $\Om:=(-1,1)\times (0,1)$.
Define $u\in\BV(\Om)$ by
\[
u(x_1,x_2):=
\begin{cases}
0 & \textrm{when }x_1<0\\
1 & \textrm{when }x_1\ge 0.
\end{cases}
\]
Let $v\in\BV(\Om)$ with $\mathcal H(S_v)=0$. Since
$\mathcal H(\{0\}\times (0,1))>0$, it is now easy to check that
$\Vert v-u\Vert_{L^{\infty}(\Om\setminus \Om(r))}\ge 1/2$ for all $r>0$.
\end{example}

A direct consequence of Corollary \ref{BV-Sobolev} and
Corollary \ref{SObolev-M-2} is that under a proper setting, the trace
spaces of the BV space, Newton-Sobolev space, and Haj\l asz-Sobolev
space are the same. Hence we can obtain many trace results for the
BV and Haj\l asz-Sobolev space directly from trace results for the
Newton-Sobolev space obtained in the literature.
{In particular, from \cite[Theorem 1.1]{Ma} we are able to obtain
the following result.}

\begin{theorem}\label{application-1}
Let $\Omega\subset X$ be a
{bounded uniform domain
satisfying the measure doubling condition \eqref{measure-doubling}.
Assume also that $(\Om,d,\mu|_{\Omega})$}
admits a $(1,1)$-Poincar\'e inequality. Let $\bx$
be endowed with an Ahlfors codimension $\theta$-regular measure $\nu$ for
some $0<\theta<1$. Then the trace spaces of $N^{1, 1}(\Omega, \mu)$,
$\BV(\Omega, \mu)$ and $M^{1,1}(\Omega, \mu)$ are the same, namely the Besov
space $B^{1-\theta}_{1,1}(\bx, \nu)$.
\end{theorem}

We say that $\bx$ is endowed with an Ahlfors codimension $\theta$-regular
measure $\nu$ if there is a $\sigma$-finite Borel measure $\nu$ on $\bx$ and
a constant $c_\theta>0$ such that
\begin{equation}\label{codimension}
{c_\theta}^{-1} \frac{\mu(B(x, r)\cap \Omega)}{r^\theta}\leq \nu(B(x, r)\cap \bx) \leq c_\theta \frac{\mu(B(x, r)\cap \Omega)}{r^\theta}
\end{equation}
for all $x\in \bx$ and $0<r<2\diam \Omega$.
{The Besov space $B^{1-\theta}_{1,1}(\bx, \nu)$}
consists of $L^p$-functions
of finite Besov norm that is given by
\[
\|u\|_{B^{1-\theta}_{1,1}(\bx, \nu)} =\|u\|_{L^p(\bx, \nu)}+\int_0^\infty\int_\bx\vint{B(y, t)}\frac{|u(x)-u(y)|}{t^{1-\theta}}\, d\nu(x)\, d\nu(y)\frac{dt}{t}.
\]

The above theorem seems to be new even for $\BV$ and $M^{1,1}$ functions
in the (weighted) Euclidean setting.
{As an illustration, we give an example in weighted Euclidean spaces.}

\begin{example}
Let $\Omega=\mathbb D \subset \R^2$ be the unit disk with $\bx=\mathbb S^1$ the
unit circle. Take the measure $d\mu (x)=\dist(x, \mathbb S^1)^{-\alpha}\,dm_2(x)$
with $0<\alpha<1$ and $m_2$ two-dimensional Lebesgue measure. Then by a
direct computation, $\dist(x, \mathbb S^1)^{-\alpha}$ with $0<\alpha<1$
is an $A_1$-weight and
hence $\mu$ supports a $(1,1)$-Poincar\'e inequality,
see \cite[Chapter 15]{HKM}.
Moreover, it is easy to check that the 1-dimensional Hausdorff measure
$\mathcal H^1$ on $\mathbb S^1$ is an Ahlfors codimension
$(1-\alpha)$-regular
measure, i.e., $\mathcal H^1$ on $\mathbb S^1$ satisfies \eqref{codimension}
with $\theta=1-\alpha$. Hence we obtain from
Theorem \ref{application-1} that the trace spaces of $N^{1,1}(\mathbb D, \mu)$,
$\BV(\mathbb D, \mu)$,
and $M^{1,1}(\mathbb D, \mu)$ are $B^{\alpha}_{1,1}(\mathbb S^1, \mathcal H^1)$.
It is also known
from the classical trace results of weighted
	Sobolev spaces that the trace space of $N^{1,1}(\mathbb D, \mu)$ is the classical
	Besov space $B^{\alpha}_{1,1}(\mathbb S^1, \mathcal H^1)$.
Here we refer to \cite{MirRus,Tyu1,Tyu2} for the trace results for weighted Sobolev
spaces on Euclidean spaces and refer to the seminal monographs by Triebel \cite{T}
for more information on Besov spaces. 
\end{example} 

{On the other hand, using our theory it is also possible to
obtain new trace results for Haj\l asz-Sobolev or Newton-Sobolev functions
from the known trace results for BV functions. In particular,
from \cite[Corollary 1.4]{MShS}
we are able to obtain the following trace results.}

\begin{theorem}\label{application-2}
Let $\Omega\subset X$ be a bounded uniform domain that satisfies the measure density
{condition \eqref{measure-density}
and admits a $(1,1)$-Poincar\'e inequality.
Assume also that the codimension
$1$ Hausdorff measure $\mathcal H$ is Ahlfors
codimension $1$-regular.}
Then we have that the trace spaces of
$\BV(\Omega,\mu)$, $N^{1,1}(\Omega,\mu)$ and $M^{1,1}(\Omega,\mu)$ are the same,
namely the space $L^1(\bx,\mathcal H)$.
\end{theorem}

When $\Omega=\mathbb D$, $\bx=\mathbb S^1$, $\mu=m_2$ the
$2$-dimension Lebesgue measure and
{$\mathcal H\approx \mathcal H^1$ the}
$1$-dimension Hausdorff measure,
the above theorem coincides with the classical results that the trace spaces
of $\BV(\mathbb D)$ and $N^{1,1}(\mathbb D)$ are both $L^1(\mathbb S^1)$. Moreover,
the above theorem gives that $L^1(\mathbb S^1)$ is also the trace space
of $M^{1,1}(\mathbb D)$, which seems to be new even in this case.

The above Theorem \ref{application-1} and Theorem \ref{application-2} both require
that the boundaries are endowed with some codimension Ahlfors regular measure.
In the following, we will give an example where the measure on the boundary
do not satisfy any codimension Ahlfors regularity.

\begin{example} \label{ex:weight}
Let $\Omega=\R^{2}_+:=\{x=(x_1, x_2)\in \R^2:\, x_2>0 \}$ and take
the measure $d\mu(x)=w_\lambda(x)\, dm_2(x)$ with $m_2$ the
$2$-dimensional Lebesgue measure and
$w_\lambda(x)= \log^\lambda\left(\max\{e, e/|x_2|\}\right)$, $\lambda>0$. 
For any $x\in \R=\partial \Omega$ and $0<r<e^{-2\lambda}$, let $Q(x, r)$
denote the cube parallel to the coordinate axes with center $x$ and
sidelength $r$. Then we have the estimate
\begin{equation}\label{A_p}
\mu(Q(x, r))= 
 2\int_{0}^{r}\int_{0}^{r/2} \log^\lambda(e/|x_2|)\, dx_2\, dx_1
= 2 r \int_{0}^{r/2} \log^\lambda(e/t)\, dt \approx r^2\log^\lambda(e/r). 
\end{equation}
Here the last equality holds since we have
\[
\left(t\log^\lambda(e/t)\right)^{'}=\log^\lambda(e/t)\left(1-\frac{\lambda}{\log(e/t)}\right)\approx \log^\lambda(e/t), \ \ \text{for} \ \  0<t\leq r<e^{-2\lambda}.
\]

By using the estimate \eqref{A_p}, it follows from the definition of
the codimension $1$ Hausdorff measure \eqref{def-codimension-1} that
for any nonempty interval $[a, b]$ in $\mathbb R=\partial \R^{2}_+$,
we have that
\[\mathcal H([a, b]) = \lim_{R\rightarrow 0^+} \mathcal H_R([a, b])\approx \lim_{R\rightarrow 0^+} |a-b| \log^{\lambda}(e/R)=\infty.
\]
Hence the codimension $1$ Hausdorff measure $\mathcal H$ on $\R$ is not
even $\sigma$-finite and is not the $1$-dimensional Lebesgue measure
that we usually study.

Moreover, the weight $w_\lambda$ defined above is a Muckenhoupt $A_1$-weight, since it is easy to check from estimate \eqref{A_p} that
$$\frac{\mu(B(z, r))}{r^2}\lesssim \inf_{x\in B(z, r)} w_\lambda(x),\ \ \text{for any} \ z\in \R^2_+ \ \text{and} \ r>0. $$
We refer to \cite{B01} and \cite[Chapter 15]{HKM} for definitions, properties and examples of Muckenhoupt class weights.
\end{example}

\begin{example}\label{dyadic-trace}
Let $\Omega, \mu$ be as in the above example. Then it is easy to check from estimate \eqref{A_p} that the $1$-dimensional Lebesgue measure on $\mathbb R$ does
not satisfy the condition \eqref{codimension} for any $\theta$. 
We denote by $\dyadic$ the collection of dyadic semi-open intervals in $\R$,
i.e.~the intervals of the form $I := 2^{-k}\big((0,1] + m\big)$, where
$k \in \mathbb N$ and $m \in \mathbb Z$. Write $\ell(I)$ for the edge length
of $I \in \dyadic$, i.e.~$2^{-k}$ in the preceding representation, and
$\dyadic_{k}$ for the cubes $Q \in \dyadic$ such that $\ell(Q) = 2^{-k}$.
For any $I\in\dyadic_{2^{j}}$, denote by $\widetilde I$ the interval in
$\dyadic_{2^{j-1}}$ containing the interval $I$.
By applying the methods used in \cite{Tyu2} and \cite[Theorem 1.3]{KW19}, we are
able to use the dyadic norm similar with the ones used in \cite{KSW17} and
\cite{KW19} to characterize the trace space of $N^{1,1}(\mathbb R^2_+, \mu)$,
which is the Besov-type space $\mathcal B^{\lambda}_{1}(\R)$. The Besov-type
space $\mathcal B^{\lambda}_{1}(\R)$ consists of functions in $L^1(\R)$ of
finite dyadic norm that is given by
\[\|u\|_{\mathcal B^{\lambda}_{1}(\mathbb R)}=\|u\|_{L^1(\mathbb R)}+\sum_{j=1}^{+\infty}2^{-\lambda j}\sum_{I\in\dyadic_{2^j}} 2^{-2^{j}} |u_{I}-u_{\widetilde I}|.
\] 
We omit the detailed proof here. Since $\R^2_+$ is uniform domain and satisfies
the measure doubling condition \eqref{measure-doubling}, hence we obtain that the
trace spaces of $\BV(\R^2_+,\mu)$, $N^{1,1}(\R^2_+,\mu)$ and $M^{1,1}(\R^2_+, \mu)$
are the same, the Besov-type space $\mathcal B^{\lambda}_{1}(\R)$.
\end{example}

\begin{example}
The recent papers \cite{BBGS,KW19,W} studied trace results on regular trees.
We refer to \cite[Section 2]{BBGS} or \cite[Section 2.1]{KW19} for the definition
of regular trees. It is easy to check that a regular tree is  uniform and that
it supports $(1,1)$-Poincar\'e inequality by modifying the proof in
\cite[Theorem 4.2]{BBGS} under the setting in \cite{BBGS,KW19}. Hence the trace
results  of $N^{1,1}$ in \cite{BBGS,KW19}  can be immediately applied to $\BV$
and $M^{1,1}$. We omit the detail here and leave it to the interested reader. 
\end{example}

\noindent Addresses:\\

\noindent P.L.: Institut f\"ur Mathematik\\
Universit\"at Augsburg\\
Universit\"atsstr. 14\\
86159 Augsburg, Germany\\
E-mail: {\tt panu.lahti@math.uni-augsburg.de}

\bigskip

\noindent X.L.: School of Mathematics(Zhuhai)\\
Sun Yat-Sen University\\
519082, Zhuhai, China\\
Email:{\tt lixining3@mail.sysu.edu.cn}

\bigskip

\noindent Z.W.:
Department of Mathematics and Statistics\\
University of Jyv\"askyl\"a\\
PO~Box~35, FI-40014 Jyv\"askyl\"a, Finland.\\
Email: {\tt zhuang.z.wang@jyu.fi}

\end{document}